\documentclass[12pt]{amsart}
\title{Labelled seeds and the mutation group}
\author{Alastair King}
\address{Alastair King\\Department of Mathematical Sciences\\University of Bath
 \\Claverton Down\\Bath\\Somerset\\BA2 7AY\\United Kingdom}
\email{A.D.King@bath.ac.uk}
\author{Matthew Pressland}
\address{Matthew Pressland\\Department of Mathematical Sciences\\University of Bath
 \\Claverton Down\\Bath\\Somerset\\BA2 7AY\\United Kingdom}
\email{mdpressland@bath.edu}
\date{October 13, 2016}
\usepackage[hmargin=3cm,vmargin=3.5cm]{geometry}
\usepackage{setspace}
\usepackage{amssymb}
\usepackage[colorlinks, urlcolor=Navy, linkcolor=Navy, citecolor=Navy]{hyperref}
\usepackage[abbrev,alphabetic,msc-links]{amsrefs}
\usepackage{url}
\usepackage{mathdots}
\usepackage[usenames,svgnames]{xcolor}
\usepackage{tikz}
\usepackage{tikz-cd}
\usetikzlibrary{arrows}
\usepackage{color}
\usepackage{url}
\usepackage[capitalize]{cleveref}
\crefname{defn}{Definition}{Definitions}
\crefname{eg}{Example}{Examples}
\crefname{lem}{Lemma}{Lemmas}
\crefname{prop}{Proposition}{Propositions}
\crefname{cor}{Corollary}{Corollaries}
\crefname{thm}{Theorem}{Theorems}
\crefname{rem}{Remark}{Remarks}
\crefname{conj}{Conjecture}{Conjectures}
\crefname{ex}{Exercise}{Exercises}
\crefname{prob}{Problem}{Problems}
\makeatletter
\renewcommand*\env@matrix[1][*\c@MaxMatrixCols c]{%
  \hskip -\arraycolsep
  \let\@ifnextchar\new@ifnextchar
  \array{#1}}
\makeatother
\theoremstyle{plain}
\newtheorem{thm}{Theorem}[section]
\newtheorem{lem}[thm]{Lemma}
\newtheorem{cor}[thm]{Corollary}
\newtheorem{prop}[thm]{Proposition}

\theoremstyle{definition}
\newtheorem{defn}[thm]{Definition}
\newtheorem{eg}[thm]{Example}

\newtheorem{rem}[thm]{Remark}
\newtheorem{prob}[thm]{Problem}
\theoremstyle{remark}

\newcommand{\noop}[1]{}

\newcommand{\set}[1]{\{#1\}}
\newcommand{\semidir}{\ltimes}
\newcommand{\defeq}{:=}
\newcommand{\eqdef}{=:}
\newcommand{\id}[1]{1_{#1}}
\newcommand{\incl}{\hookrightarrow}
\newcommand{\lquot}{\backslash}
\newcommand{\isom}{\cong}
\newcommand{\isoto}{\to}
\newcommand{\rest}[2]{#1|_{#2}}

\newcommand{\cA}{\mathcal{A}}
\newcommand{\bb}{\mathbf{b}}
\newcommand{\cF}{\mathcal{F}}
\newcommand{\cG}{\mathcal{G}}
\newcommand{\cK}{\mathcal{K}}
\newcommand{\QQ}{\mathbb{Q}}
\newcommand{\cS}{\mathcal{S}}
\newcommand{\ZZ}{\mathbb{Z}}

\newcommand{\Sym}[1]{\mathfrak{S}_{#1}}
\newcommand{\typA}[1]{\mathsf{A}_{#1}}
\newcommand{\afftypA}[1]{\widetilde{\mathsf{A}}_{#1}}
\newcommand{\transp}{^\mathsf{T}}
\newcommand{\tran}[2]{({#1}\;{#2})}
\newcommand{\from}{\leftarrow}
\newcommand{\maut}[2]{\alpha_{#1}^{#2}}
\newcommand{\waut}[1]{\alpha_{#1}}
\newcommand{\morph}[2]{\langle#2,#1\rangle}
\newcommand{\lsgraph}[1]{\Delta(#1)}
\newcommand{\norm}[2]{N_{#1}(#2)}
\newcommand{\normgpd}[2]{N^{#1}_{#2}}
\newcommand{\orbgpd}{/\hspace{-3pt}/}
\newcommand{\act}[2]{\varphi_{#1}^{#2}}
\newcommand{\cW}{W^{\mathrm{c}}}
\newcommand{\dcW}{W^+}

\DeclareMathOperator{\op}{op}
\DeclareMathOperator{\Aut}{Aut}
\DeclareMathOperator{\Stab}{Stab}
\DeclareMathOperator{\Iso}{Iso}

\DeclareMathOperator{\Hom}{Hom}

\DeclareMathOperator{\CMG}{CMG}

\DeclareMathOperator{\im}{im}
\DeclareMathOperator{\cAut}{Aut^{\mathrm{c}}}
\DeclareMathOperator{\dircAut}{Aut^+}
\begin{document}

\begin{abstract}
We study the set $\cS$ of labelled seeds of a cluster algebra of rank $n$ inside a field $\cF$ 
as a homogeneous space for the group $M_n$ of (globally defined) mutations and relabellings. Regular equivalence relations on $\cS$ are associated to subgroups $W$ of $\Aut_{M_n}(\cS)$, 
and we thus obtain groupoids $W\lquot\cS$. 
We show that for two natural choices of equivalence relation, the corresponding groups $\cW$ and $\dcW$ act on $\cF$, and the groupoids $\cW\lquot\cS$ and $\dcW\lquot\cS$ act on the model field $\cK=\QQ(x_1,\dotsc,x_n)$. 
The groupoid $\dcW\lquot\cS$ is equivalent to Fock--Goncharov's cluster modular groupoid. 
Moreover, $\cW$ is isomorphic to the group of cluster automorphisms, and $\dcW$ to the subgroup of direct cluster automorphisms, in the sense of Assem--Schiffler--Shramchenko.

We also prove that, for mutation classes whose seeds have mutation finite quivers, the stabiliser of a labelled seed under $M_n$ determines the quiver of the seed up to `similarity', meaning up to taking opposites of some of the connected components. Consequently, the subgroup $\cW$ is the entire automorphism group of $\cS$ in these cases.
\end{abstract}
\maketitle

\section*{Introduction}
A cluster algebra, defined by Fomin and Zelevinsky in \cite{fomincluster1}, is a type of commutative algebra with a particular combinatorial structure coming from a collection of seeds,
related to each other by mutations. 
However, mutations are defined only locally and there is no group of mutations acting globally on seeds. 
In order to obtain such an action, one considers instead the larger collection of labelled seeds, on which mutations do act as a group, as do permutations of the labels. By studying labelled seeds, with this combined action of relabelling and mutation, we are able to apply the theory of groups and homogeneous spaces to cluster combinatorics. The cluster automorphism group of \cite{assemcluster} and the cluster modular groupoid of \cite{fockcluster1} both appear naturally from this point of view.

It is also natural (and not uncommon) to consider cluster algebras inside any field~$\cF$, over~$\QQ$, isomorphic to $\cK=\QQ(x_1,\dotsc,x_n)$, rather than restricting to subalgebras of $\cK$ itself. 
Each labelled seed then includes a choice of labelled free generating set in $\cF$, or equivalently an isomorphism $\cK\isoto\cF$.

The structure of the paper is as follows. 
In \Cref{labseedgraph}, we recall the definitions of cluster algebras and labelled seeds, and see that the collection of labelled seeds forms a homogeneous space for the mutation group $M_n$, which is  the semidirect product of the free group on $n$ involutions with the symmetric group on $n$ elements. 
One goal will be to study the automorphism group of this homogeneous space.

In \Cref{regequivrels}, we describe the more general theory of regular equivalence relations on homogeneous spaces, and explain the `Galois correspondence' between such relations and subgroups of the automorphism group of the homogeneous space. 
Such a subgroup 
gives rise to an orbit groupoid, defined in \Cref{orbitgroupoid}, whose objects are the orbits under the action of the subgroup.

In \Cref{clustmodgpd}, we return to the setting of labelled seeds, and consider two particular regular equivalence relations, one relating labelled seeds with the same quiver and the other relating labelled seeds with `similar' quivers,
that is, quivers that are the same up to taking the opposite of some connected components. 
By defining a functor on the resulting orbit groupoids, we obtain an action of each groupoid on the model field~$\cK$. We also show that, for the equivalence relation of having the same quiver, the orbit groupoid is equivalent to Fock and Goncharov's cluster modular groupoid \cite{fockcluster1}.

In \Cref{clustautgp}, we show how the subgroups corresponding to our two special equivalence relations act on $\cF$, with elements acting as cluster automorphisms in the sense of \cite{assemcluster}. 
We also observe that the isomorphism $\cK\isoto\cF$, corresponding to any labelled seed, intertwines the action of the orbit groupoid on $\cK$ with the action of the group on $\cF$. We conclude the section by showing that the group corresponding to the relation of having the same quiver is isomorphic to the group $\dircAut{\cA}$ of direct cluster automorphisms. The group $\cAut{\cA}$ of all cluster automorphisms corresponds to the weaker relation of having similar quivers.

We prove, in \Cref{stabsdeterminequivs}, that when the number of quivers occurring in a mutation class of labelled seeds is finite, the relation of having similar quivers is the same as the relation of having the same stabiliser under the mutation action, which is the relation corresponding to the entire automorphism group of the homogeneous space of labelled seeds. Thus in such classes, the stabiliser of a seed determines the similarity class of its quiver.

Finally, in \Cref{orbitgpegs}, we give examples of the sets of equivalence classes of labelled seeds under our two chosen equivalence relations for various explicit quivers. In particular we show that the conclusions of \Cref{stabsdeterminequivs} need not hold in mutation classes with infinitely many quivers.

\section*{Acknowledgements}
We would like to thank Thomas Br\"ustle for drawing our attention to Assem--Schiffler--Shramchenko's paper on cluster automorphisms \cite{assemcluster}, and the referee for helpful suggestions. This paper was written during the second author's Ph.D.\ studies, which were supported by an Engineering and Physical Sciences Research Council studentship.

\section{Cluster algebras and labelled seeds}
\label{labseedgraph}
Cluster algebras were first defined in terms of seeds in \cite{fomincluster2}.
We will consider only a restricted class of cluster algebras, namely those of geometric type, without frozen variables and with skew-symmetric exchange matrices.
Such algebras can be defined as follows.

Let $\cF$ be a purely transcendental field extension of $\QQ$ of transcendence degree $n$. 
A \emph{free generating set} of $\cF$
is a set $\bb\subset\cF$ (necessarily with $n$ elements) such that each element of $\cF$ can be written uniquely as a rational function in the elements of $\bb$, with coefficients in $\QQ$,
i.e.\ $\bb$ is a transcendence basis that generates $\cF$ (cf.\ \cite{berensteinquantum}*{Rem.~2.2}).

A \emph{seed} consists of a free generating set $\bb$ for $\cF$
and a quiver $Q$, with vertex set $\bb$, and without loops or $2$-cycles.
Given $v\in\bb$, we can construct a new seed, called the \emph{mutation} of the original seed at $v$, by replacing $v$ by $v'\in\cF$ satisfying
\[vv'=\prod_{w\in\bb}w^{a_{vw}}+\prod_{w\in\bb}w^{a_{wv}},\]
where $A=(a_{uw})_{u,w\in\bb}$ is the adjacency matrix of $Q$, while replacing $Q$ by its Fomin--Zelevinsky mutation at $v$. 
If we mutate this new seed at $v'$, we recover the original seed. 
Given an initial seed, the elements of all free generating sets belonging to the seeds that can be obtained by iterated mutation from the initial seed are called \emph{cluster variables}.
The subalgebra $\cA\subset\cF$ they generate is a \emph{cluster algebra} of rank~$n$.

Note that, in our context, a quiver comprises no more data than its adjacency matrix.
In particular, an isomorphism of quivers is for us simply a bijection between the vertex sets which preserves the arrow multiplicity between each ordered pair of vertices.
When a quiver has no loops or $2$-cycles, its adjacency matrix $A$ is uniquely determined by the skew-symmetric \emph{exchange matrix} $A-A\transp$. More detailed background information on cluster algebras from quivers can be found in Keller's survey article \cite{kellercluster}.

The automorphism group $\Aut{\cF}$ acts on seeds; an automorphism $\alpha$ replaces each vertex 
variable $x$ by $\alpha(x)$ and leaves the arrows fixed. 
Given a cluster algebra $\cA\subset\cF$, equipped with its collection of seeds, an element $\alpha\in\Aut{\cF}$ is 
a \emph{cluster automorphism} of $\cA$ if for every seed $(Q,\bb)$ of $\cA$, the seed $(Q,\alpha(\bb))$ is also a seed of $\cA$, 
up to reversing the orientation of some connected components of $Q$
(see \cite{assemcluster}*{Defn.~1, Lem.~2.3, Prop.~2.4}).
We say that $\alpha$ is \emph{direct} if $(Q,\alpha(\bb))$ is actually a seed of $\cA$, 
without any reversal of orientation.
Note that any cluster automorphism $\alpha$ permutes the cluster variables of $\cA$ \
and thus restricts to an algebra automorphism of $\cA$. 
We denote the group of cluster automorphisms by $\cAut{\cA}$ and the subgroup of direct cluster automorphisms by $\dircAut{\cA}$.

As defined above, mutation is only defined locally at each seed. However,  it can be convenient to label seeds with some indexing set $I$, so that there is a globally defined mutation operation $\mu_i$ for each $i\in I$.
This variation was introduced implicitly in  \cite{berensteincluster3}*{\S1.1} and explicitly in \cite{fomincluster4}*{Defn.~2.3};
see also \cite{kellercluster}, \cite{reitencluster}.
For simplicity, we take $I=\set{1,\dotsc,n}$, but note that the natural ordering of this set plays no role and any other $n$-element set would do equally well.

\begin{defn}
A \emph{labelled seed} is a pair $(Q,\beta)$ in which $Q$ is a quiver with vertex set $Q_0=I$, and $\beta=(\beta_i : i\in I)\in\cF^I$ freely generates $\cF$.
\end{defn}

We can define the \emph{mutation of $(Q,\beta)$ at $i$} as $(Q,\beta)\cdot\mu_i=(Q',\beta')$,
where the quiver $Q'$ is given by the Fomin--Zelevinsky mutation of $Q$ at the vertex $i$ and
\[
  \beta'_j=\begin{cases}\beta_j,&j\ne i,\\\frac{\prod_{k\in I}\beta_k^{a_{ki}}+\prod_{k\in I}\beta_k^{a_{ik}}}{\beta_i},&j=i,\end{cases}  
\]
where $A=(a_{ij})_{i,j\in I}$ is the adjacency matrix of $Q$. 
If two labelled seeds differ by simultaneous relabelling of the quiver vertices and the free generating set, then they determine the same (unlabelled) seed.
This relabelling may be formulated as a free right action of the symmetric group $\Sym{n}$ on the set of labelled seeds: for a permutation $\sigma\in\Sym{n}$ we define $(Q,\beta)\cdot\sigma=(Q^\sigma,\beta^\sigma)$, where $\beta^\sigma_i=\beta_{\sigma(i)}$ and $Q^\sigma$ is the quiver with vertex set $I$ and adjacency matrix $a^{\sigma}_{ij}=a_{\sigma(i)\sigma(j)}$.

\begin{defn}
The relabelling action of $\Sym{n}$ combines with the mutation action of the $\mu_i$ 
to give a right action of the \emph{mutation group}
\begin{equation*}\label{mutationgroup}
  M_n:=\Sym{n}\semidir\langle\mu_1,\dotsc,\mu_n:\mu_i^2=1\rangle
\end{equation*}
on the set of labelled seeds. It is a semi-direct product because, for any labelled seed $(Q,\beta)$, any $i\in I$ and any $\sigma\in\Sym{n}$, we have
\[
  ((Q,\beta)\cdot\sigma)\cdot\mu_i=((Q,\beta)\cdot\mu_{\sigma(i)})\cdot\sigma.
\]
\end{defn}

The inclusion of $\Sym{n}$ in $M_n$ is the key to recovering the usual cluster combinatorics of unlabelled seeds via the group action.
In particular, we can define the \emph{mutation class} of a labelled seed $(Q,\beta)$ simply to be its orbit $\cS$ under the $M_n$-action and consider that such an orbit constitutes a \emph{cluster structure} on $\cF$: the clusters are the sets $\set{\beta_1,\dotsc,\beta_n}$ for each $(Q,\beta)\in\cS$, and, as usual, the set of cluster variables is the union of all clusters and the cluster algebra $\cA(\cS)$ is the subalgebra of $\cF$ generated by all cluster variables. Note that $\cA(\cS)$ can also be described as the cluster algebra generated in the usual way by the unlabelled seed naturally attached to any labelled seed of $\cS$.
We also note that a cluster structure $\cS$ is a homogeneous space for $M_n$, an observation that motivates our approach in this paper.

Another benefit to considering the set of labelled seeds is that the action of $\langle\mu_i : i\in I \rangle$ can be encoded by the structure of a labelled graph, which we will denote by $\lsgraph{\cF}$: its vertices are the labelled seeds of $\cF$,  two of which are joined by an edge labelled $i\in I$ when they are related by mutation at $i$. 
In other words, $\mu_i$ acts by interchanging all pairs of vertices that are adjacent along an edge labelled $i$.
The graph $\lsgraph{\cF}$ is $n$-regular as a labelled graph, i.e.\ each vertex is incident with exactly one edge labelled $i$ for each $i\in I$. 
The action of $\Sym{n}$, taking a seed $(Q,\beta)$ to $(Q^\sigma,\beta^\sigma)$, is equivariant, i.e.\ edges labelled $\sigma(i)$ are taken to edges labelled $i$.
Thus the quotient of $\lsgraph{\cF}$ by $\Sym{n}$ is not a labelled graph.

Given a cluster structure on $\cF$, that is, a mutation class $\cS$ of labelled seeds,
we will denote the full subgraph of $\lsgraph{\cF}$ on $\cS$ by $\lsgraph{\cS}$. 
This graph is not necessarily connected, but, by construction, the mutations $\langle\mu_i\rangle$ act transitively on each connected component. 
Thus each component is a quotient of the Cayley graph of $\langle\mu_i\rangle$, that is, the $n$-regular labelled tree. 
The quotient of $\lsgraph{\cS}$ by the $\Sym{n}$-action is connected and is the \emph{cluster exchange graph},
that is, its vertices are unlabelled seeds and its (unlabelled) edges correspond to mutations between them.

\begin{eg}
\label{a2eg}
Take $\cF=\QQ(x,y)$. The cluster algebra of type $\typA{2}$ with initial (unlabelled) seed $x\to y$ is the subalgebra $\QQ[x,y,\frac{1+x}{y},\frac{1+y}{x},\frac{1+x+y}{xy}]$ of $\cF$. 
Let $s_1$ be the labelled seed
\[s_1=\Big(1\to 2,(x,y)\Big).\]
The orbit $\cS$ of $s_1$ under $M_n$ consists of the $10$ labelled seeds
\begin{align*}
s_1=&\Big(1\to 2,(x,y)\Big),&s_6=&\Big(1\leftarrow2,(y,x)\Big),\\
s_2=&\Big(1\leftarrow 2,\big(\tfrac{1+y}{x},y\big)\Big),&s_7=&\Big(1\to2,\big(y,\tfrac{1+y}{x}\big)\Big),\\
s_3=&\Big(1\to 2,\big(\tfrac{1+y}{x},\tfrac{1+x+y}{xy}\big)\Big),&s_8=&\Big(1\leftarrow2,\big(\tfrac{1+x+y}{xy},\tfrac{1+y}{x}\big)\Big),\\
s_4=&\Big(1\leftarrow 2,\big(\tfrac{1+x}{y},\tfrac{1+x+y}{xy}\big)\Big),&s_9=&\Big(1\to2,\big(\tfrac{1+x+y}{xy},\tfrac{1+x}{y}\big)\Big),\\
s_5=&\Big(1\to 2,\big(\tfrac{1+x}{y},x\big)\Big),&s_{10}=&\Big(1\leftarrow2,\big(x,\tfrac{1+x}{y}\big)\Big),
\end{align*}
indexed by $\ZZ_{10}$.
The seed $s_i$ is related to $s_{i+1}$ by mutation at $1$ (for odd $i$) or $2$ (for even $i$).
So $\Delta(\cS)$ is the graph
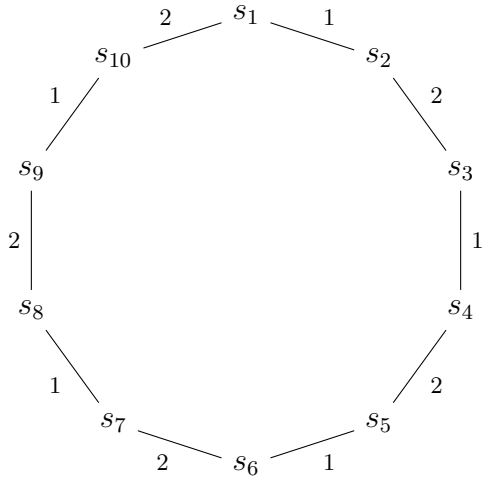
\[ 
\begin{tikzpicture}[scale=1.5]
\foreach \n in {1,...,10}
 \node at (126-36*\n:2) (s\n) {$s_{\n}$};
\foreach \a/\b in {1/2, 3/4, 5/6, 7/8, 9/10}
\path[font=\scriptsize] (s\a) edge node [auto] {$1$} (s\b);
\foreach \a/\b in {2/3, 4/5, 6/7, 8/9, 10/1}
\path[font=\scriptsize] (s\a) edge node [auto] {$2$} (s\b);
\end{tikzpicture}
\] 
The group $M_2$ acts on the vertex set with the involution $\mu_i$ exchanging vertices adjacent along an edge labelled by $i$
and the transposition $\tran12$  taking $s_i$ to $s_{i+5}$. 
Note that $(\mu_1\mu_2)^5\in M_2$ fixes all vertices, so the action is not faithful.
Indeed, a complete set of relations is
\[ \mu_1\mu_2\mu_1\mu_2\mu_1 = \tran12 =  \mu_2\mu_1\mu_2\mu_1\mu_2.\]

The quotient by the $\Sym{2}$-action recovers the cluster exchange graph of type $\typA{2}$, the familiar pentagon (\cite{kellercluster}*{\S3.5}). It is not possible to consistently label the edges of the pentagon.
\end{eg}

Another reason for considering labelled seeds, which will play an important role in this paper, is that while roughly speaking a free generating set for $\cF$ determines an isomorphism with a model field, this is only actually correct if the set is labelled. 
Thus, if $\cK= \QQ(x_1,\dotsc,x_n)$, then a labelled cluster $\beta=(\beta_1,\dotsc,\beta_n)$ determines, and is determined by, an isomorphism  $\cK \to \cF$, which we also denote by $\beta$, with $\beta(x_i)=\beta_i$. 

From this point of view, there is another way of describing the right action of $M_n$ on the set of labelled seeds.
Let $\maut{\mu_i}{Q}\colon\cK\to\cK$ be the automorphism
\[
  \maut{\mu_i}{Q}(x_j)=\begin{cases}x_j,&j\ne i,\\\frac{\prod_{k\in I}x_k^{a_{ki}}+\prod_{k\in I}x_k^{a_{ik}}}{x_i},&j=i.\end{cases}
\]
If $(Q',\beta')=(Q,\beta)\cdot\mu_i$, then there is an equality $\beta' = \beta\circ\maut{\mu_i}{Q}$ of maps $\cK\to \cF$, 
i.e.\ the map $\maut{\mu_i}{Q}$ is a `change of coordinates' depending on $Q$.
For $\sigma\in\Sym{n}$, let $\maut{\sigma}{}\colon\cK\to\cK$ be the coordinate permutation  $\maut{\sigma}{}(x_i)=x_{\sigma(i)}$.
Then we also have $\beta^\sigma=\beta\circ\maut{\sigma}{}$. 
It will be convenient to write $\maut{\sigma}{Q}\defeq\maut{\sigma}{}$, even though this automorphism does not depend on the quiver.

More generally, for any $g\in M_n$, write $g=g_1\dotsm g_k$, where each $g_i$ is either a mutation or a permutation, and define
\[
  \maut{g}{Q}=\maut{g_1}{Q}\circ\maut{g_2}{Q\cdot g_1}\circ\dotsb\circ\maut{g_k}{Q\cdot g_1\dotsm g_{k-1}}.
\]
If $(Q',\beta') = (Q,\beta)\cdot g$, then $\beta' =\beta\circ\maut{g}{Q}$. 
Hence $\maut{g}{Q}=\beta^{-1}\circ\beta'\colon \cK\to \cK$ and so it is independent of the chosen expression for $g$.

There is a natural action of $\Aut{\cF}$ on labelled seeds, analogous to the action on seeds described above, by
$\alpha\colon (Q,\beta) \mapsto (Q,\alpha\circ\beta)$, for $\alpha\in\Aut{\cF}$.
This action commutes with the $M_n$-action and thus induces an action on $M_n$-orbits, i.e.\ cluster structures.
Let $\cS$ be a cluster structure and let $\cA(\cS)$ be the associated cluster algebra. If $(Q,\bb)$ is any seed of $\cA(\cS)$, we can produce a labelled seed $(Q,\beta)$ by arbitrarily labelling the elements of $\bb$ as $\beta_1,\dotsc,\beta_n$ and replacing the vertex $\beta_i$ of $Q$ by $i$. The $n!$ possible relabellings are related by the $\Sym{n}$-action, so all lie in $\cS$. It follows that $(Q,\alpha(\bb))$ is a seed of $\cA(\cS)$ if and only if $(Q,\alpha\circ\beta)\in\cS$ for any associated labelled seed $(Q,\beta)$, and so the stabiliser of $\cS$ under the $\Aut{\cF}$-action is the direct cluster automorphism group $\dircAut{\cA(\cS)}$.
Similar arguments show that $\cAut{\cA(\cS)}$ is the stabiliser of the union of all cluster structures whose labelled seeds have quivers related to those of the labelled seeds of $\cS$ by reversing the orientation of some collection of connected components.
Later (\Cref{clustautthm}) we will give alternative descriptions of $\cAut{\cA(\cS)}$ and $\dircAut{\cA(\cS)}$, each depending only on the structure of the single cluster structure $\cS$ as a homogeneous space equipped with particular equivalence relations.

The equivalence relations we will consider have the property that their equivalence classes are orbits of subgroups of the automorphism group $\Aut_{M_n}(\cS)$ of the homogeneous space $\cS$. We now give a characterisation of such equivalence relations, and describe some constructions arising from them.

\section{Regular equivalence relations}
\label{regequivrels}
In this section and the next, we will work in the setting of a general homogeneous space, rather than the set of labelled seeds discussed in \Cref{labseedgraph}. Let $X$ be a homogeneous space for a group $G$, acting on the right. Then we define the \emph{automorphism group} of $X$, denoted $\Aut_G(X)$, to be the group of bijections $\varphi\colon X\to X$ that commute with the action of $G$. We will take composition in $\Aut_G(X)$ to be right-to-left, so that $\Aut_G(X)$ acts naturally on the left of $X$ by $\varphi\cdot x=\varphi(x)$, and
\[(\varphi\cdot x)\cdot g=\varphi\cdot(x\cdot g)\]
for all $\varphi\in\Aut_G(X)$, $x\in X$ and $g\in G$.

In this section, we define homogeneity and regularity of equivalence relations on $X$, and demonstrate a Galois correspondence between regular equivalence relations on $X$ and subgroups of $\Aut_G(X)$. The results of this section are analogous to results on regular coverings in topology (see \cite{fultonalgebraic}*{\S13b}).
\begin{defn}
\label{homequiv}
An equivalence relation $\sim$ on $X$ is \emph{homogeneous} if $x\cdot g\sim y\cdot g$ for any $g\in G$ whenever $x\sim y$.
\end{defn}
\begin{defn}
\label{regequiv}
An equivalence relation $\sim$ on $X$ is \emph{regular} if it is homogeneous and $\Stab_G(x)=\Stab_G(y)$ whenever $x\sim y$.
\end{defn}
\begin{defn}
\label{normalequiv}
An equivalence relation $\sim$ on $X$ is \emph{normal} if it is regular and $\varphi(x)\sim\varphi(y)$ for any $\varphi\in\Aut_G(X)$ whenever $x\sim y$.
\end{defn}

In the case of mutation classes of labelled seeds, we will be particularly interested in the following two regular equivalence relations.

\begin{eg}
\label{regequivdefs}
Let $\cS$ be a mutation class of labelled seeds and define an equivalence relation $\simeq$ on $\cS$ such that $(Q_1,\beta_1)\simeq (Q_2,\beta_2)$ if and only if $Q_1=Q_2$; as all of our quivers have the same vertex set, we say $Q_1=Q_2$ when the quivers have identical adjacency matrices. Note that this condition is stronger than requiring only that $Q_1$ and $Q_2$ are isomorphic. This relation is homogeneous, and if $(Q,\beta)\cdot g=(Q,\beta)$ for some $g\in M_n$, then $\maut{g}{Q}=\id{\cK}$, and so $(Q,\gamma)\cdot g=(Q,\gamma)$ for all $\gamma$. Thus $\simeq$ is regular.

We can define another equivalence relation $\approx$ on $\cS$ such that $(Q_1,\beta_1)\approx(Q_2,\beta_2)$ if and only if $Q_1$ is obtained from $Q_2$ by reversing the orientation of all arrows in a set of components; we will also write $Q_1\approx Q_2$ in this case. As $(Q\cdot\mu_i)^{\op}=Q^{\op}\cdot\mu_i$ and $(Q^{\op})^\sigma=(Q^\sigma)^{\op}$ for any quiver $Q$ and vertex $i$, and $\maut{g}{Q_1}=\maut{g}{Q_2}$ when $Q_1\approx Q_2$, this equivalence relation is also regular. If $(Q_1,\beta_1)\approx(Q_2,\beta_2)$, we say that $(Q_1,\beta_1)$ and $(Q_2,\beta_2)$ are \emph{similar}, and that $Q_1$ and $Q_2$ are similar.

To illustrate these relations, recall the mutation class $\cS$ from \Cref{a2eg}. All $10$ seeds in $\cS$ are similar as the only quivers occurring are $1\to 2$ and $1\from 2$, which are opposites of each other. These two quivers are not equal (despite being isomorphic), so the equivalence classes under $\simeq$ are $\set{s_2,s_4,s_6,s_8,s_{10}}$ and $\set{s_1,s_3,s_5,s_7,s_9}$.
\end{eg}

The largest possible regular homogeneous equivalence relation (i.e.\ that with the largest equivalence classes) on any homogeneous space $X$ for a right $G$-action is the relation declaring $x,y\in X$ to be equivalent if and only if $\Stab_G(x)=\Stab_G(y)$.

Denote the equivalence class of $x$ under $\sim$ by $[x]$. If $\sim$ is homogeneous, then $[x]=[y]$ implies $[x\cdot g]=[y\cdot g]$ for all $g\in G$, so the set $X/{\sim}$ of equivalence classes admits a natural $G$-action by $[x]\cdot g=[x\cdot g]$. Thus $X$ has the structure of a homogeneous bundle over the set $X/{\sim}$ of equivalence classes. Choosing $x\in X$ allows us to identify $X$ with the quotient $G/\Stab_G(x)$, and $X/{\sim}$ with the quotient $G/\norm{G}{[x]}$, where
\[
\norm{G}{[x]}\defeq\set{g\in G:x\cdot g\sim x}
\]
is the stabiliser of the equivalence class $[x]$ under the induced action of $G$ on $X/{\sim}$.
\begin{lem}
\label{regiffnormal}
Let $\sim$ be a homogeneous equivalence relation on $X$. Then $\sim$ is regular if and only if $\Stab_G(x)$ is normal in $\norm{G}{[x]}$ for each $x\in X$.
\end{lem}
\begin{proof}
Let $\sim$ be regular. For any $h\in\norm{G}{[x]}$, we have $x\sim x\cdot h$, so
\[\Stab_G(x)=\Stab_G(x\cdot h)=h^{-1}\Stab_G(x)h.\]
Conversely, suppose $\Stab_G(x)$ is normal in $\norm{G}{[x]}$. If $y\sim x$, then $y=x\cdot h$ for some $h\in\norm{G}{[x]}$. We have
\[\Stab_G(y)=\Stab_G(x\cdot h)=h^{-1}\Stab_G(x)h=\Stab_G(x),\]
and so $\sim$ is regular.
\end{proof}

Thus if $\sim$ is regular, the fibre $[x]$ of $X$ as a bundle over $X/{\sim}$ is a torsor for the group $\norm{G}{[x]}/\Stab_G(x)$. We will see that these subquotients are all (non-canonically) isomorphic to the same subgroup of $\Aut_G(X)$, because there is in fact a Galois correspondence between regular equivalence relations and subgroups of $\Aut_G(X)$ (cf.\ \cite{fultonalgebraic}*{\S13d}).

\begin{prop}
\label{regequivlem}
Regular equivalence relations on $X$ are in one-to-one correspondence with subgroups of $\Aut_G(X)$, in such a way that the equivalence classes are the orbits of the corresponding subgroup. 
A regular equivalence relation $\sim$ is normal if and only if the corresponding subgroup is normal in $\Aut_G(X)$.
\end{prop}

\begin{proof}
Note first that $\Aut_G(X)$ acts freely on $X$, because $G$ acts transitively. 
Also, for any $w\in\Aut_G(X)$, we have $\Stab_G(w\cdot x)=\Stab_G(x)$, because, for any $g\in G$,
\[
  w\cdot x =(w\cdot x)\cdot g=w\cdot(x\cdot g) \iff x= x\cdot g.
\]
Thus, for any $W\leq\Aut_G(X)$, its orbits are the equivalence classes of a regular equivalence relation and, because the action is free, different subgroups define different equivalence relations.

To prove the correspondence, we must show that every regular equivalence relation on $X$ arises in this way.
So suppose that $\sim$ is a regular equivalence relation and let 
\[
  W = \set{w\in\Aut_G(X) : w\cdot x\sim x,\; \text{for all $x\in X$}}.
\]
By definition, for any $x\in X$, the orbit $W\cdot x$ is contained in the equivalence class $[x]$, 
so it remains show that $W$ acts transitively on each equivalence class
(cf.\ \cite{fultonalgebraic}*{Ex.~13.13}).

Fix any $x\in X$ and any $y\in [x]$.
We claim that the map $w^x_y\colon X\to X \colon x\cdot h\mapsto y\cdot h$ is a well-defined element of $W$, 
which clearly maps $x$ to $y$.
Most importantly, the fact that $w^x_y$ is a well-defined bijection (with inverse $w^y_x$) uses 
the regularity of $\sim$, noting that
\[
x\cdot h=x\cdot k \iff kh^{-1}\in\Stab_G(x)=\Stab_G(y) \iff y\cdot h=y\cdot k.
\]
Then $w^x_y\in\Aut_G(X)$, i.e. $w^x_y$ commutes with the action of $G$, as
\[
 w^x_y(x\cdot hg)=y\cdot hg=w^x_y(x\cdot h)\cdot g
\]
and it remains to show that $w^x_y\in W$, i.e. that $w^x_y(z)\sim z$, for all $z\in X$.
This follows by writing $z=x\cdot h$, so that $w^x_y(z)=y\cdot h\sim x\cdot h$, as
$\sim$ is homogeneous. 
Thus $W$ acts transitively on any equivalence class $[x]$, as required.

To prove the final part of the proposition, let $\sim$ be a regular equivalence relation and $W$ be the corresponding subgroup. The relation $\sim$ is normal if and only if
\[
x\sim y\iff\varphi(x)\sim\varphi(y),
\]
for all $\varphi\in\Aut_G(X)$. Equivalently, every $\varphi\in\Aut_G(X)$ induces bijections between equivalence classes, so $\varphi([x])=[\varphi(x)]$ for all $x\in X$. As $[x]=W\cdot x$, this is equivalent to requiring $\varphi W\cdot x=W\varphi\cdot x$ for all $x\in X$, i.e.\ that $\varphi W=W\varphi$. Thus $\sim$ is normal if and only if $W$ is normal.
\end{proof}

As promised, a standard consequence of the Galois correspondence is as follows
(cf.\ \cite{fultonalgebraic}*{Thm.~13.11} and \cite{browntopology}*{10.6.4}).

\begin{prop}
\label{Wissubquotient}
Given a regular equivalence relation $\sim$, the corresponding subgroup $W\leq\Aut_G(X)$ is isomorphic to $\norm{G}{[x]}/\Stab_G(x)$ for any $x\in X$.
\end{prop}
\begin{proof}
Note that by assumption, $\norm{G}{[x]}$ acts transitively on $[x]$. By regularity of $\sim$, all elements of $[x]$ have stabiliser $\Stab_G(x)$ under the $G$-action, so $\norm{G}{[x]}/\Stab_G(x)$ acts freely and transitively on $[x]$. As $W\leq\Aut_G(X)$, it also acts freely on $[x]$, and the action is transitive by \Cref{regequivlem}. These two actions commute because $\norm{G}{[x]}\leq G$ and $W\leq\Aut_G(X)$. Thus $[x]$ is a bitorsor for $\norm{G}{[x]}/\Stab_G(x)$ and $W$, and choosing any $y\in [x]$ yields an isomorphism between the two groups. Explicitly, having chosen $y$, we identify each $w\in W$ with the unique $g\in\norm{G}{[x]}/\Stab_G(x)$ such that $w\cdot y=y\cdot g$. This identification is well-defined and bijective because the actions are free and transitive. It is a homomorphism, because when $w_1\cdot y=y\cdot g_1$ and $w_2\cdot y=y\cdot g_2$, we have
\[w_1w_2\cdot y=w_1\cdot(y\cdot g_2)=(w_1\cdot y)\cdot g_2=y\cdot g_1g_2\]
as required.
\end{proof}

\begin{eg}
We illustrate \Cref{regequivlem} by returning to the example of type $\typA{2}$, from \Cref{a2eg}. Let $G=M_2$ and let $\cS$ be the orbit of $s_1$. Recall that $M_2$ acts on $\cS$ with $s\cdot\mu_i$ given by the unique vertex adjacent to $s$ along an edge labelled $i$ and with the transposition $\tran{1}{2}$ acting in the same way as $\mu_1\mu_2\mu_1\mu_2\mu_1$, so it will suffice to consider the action of the mutations. 

Let $\approx$ be the equivalence relation on $\cS$ with $(Q_1,\beta_1)\approx(Q_2,\beta_2)$ if and only if $Q_1\approx Q_2$, as in \Cref{regequivdefs}, and let $\cW$ be the corresponding subgroup of $\Aut_{M_2}(\cS)$. As there is only a single equivalence class under $\approx$, the group $\cW$ acts transitively, and in fact $\cW=\Aut_{M_2}(\cS)$. This automorphism group is isomorphic to $D_5$, the dihedral group of order $10$, generated by the rotation $s_i\mapsto s_{i+2}$, and the reflection interchanging $s_{1-k}$ and $s_{2+k}$ for $0\leq k\leq 4$. (Drawing $\cS$ as in \Cref{a2eg}, the reflection is a genuine reflection in the vertical axis.)

Every $s\in\cS$ has the same stabiliser under the $M_2$-action, namely the subgroup $H$ generated by $\tran{1}{2}\mu_1\mu_2\mu_1\mu_2\mu_1$ and $(\mu_1\mu_2)^5$. Therefore the mutations $\mu_1$ and $\mu_2$ generate a free right action of $M_2/H\isom D_5$. Although $\Aut_{M_2}(\cS)\isom M_2/H$, there is no natural isomorphism between the two groups; indeed, the set $\cS$ is a bitorsor for $\Aut_{M_2}(\cS)$ acting on the left and $M_2/H$ acting on the right, so each choice of $s\in\cS$ determines an isomorphism (cf.\ \Cref{Wissubquotient}).

Now let $\simeq$ be the equivalence relation on $\cS$ with $(Q_1,\beta_1)\simeq(Q_2,\beta_2)$ if and only if $Q_1=Q_2$. There are two equivalence classes, one consisting of $s_i$ for odd $i$, and the other of $s_i$ for even $i$. In this case we can see directly that these are the orbits of the action by the order $5$ cyclic subgroup $\dcW\leq\Aut_{M_2}(\cS)$ generated by the rotation.

However, we can instead follow \Cref{regequivlem}. Pick a point of $\cS$, say $s_1$, and consider the set $\norm{M_2}{[s_1]}$ of $g\in M_2$ with $s_1\cdot g\simeq s_1$. We find that $g=\mu_1\mu_2$ is such a group element, as $s_1\cdot\mu_1\mu_2=s_3\simeq s_1$. Then $g$ defines an element $w_g^{s_1}\in\Aut_{M_2}(\cS)$, with
\[w_g^{s_1}\cdot s_i=s_i\cdot g_igg_i^{-1}\]
where $g_i$ is any element of $M_2$ such that $s_i\cdot g_i=s_1$. For example, if we want to compute the action of $w_g^{s_1}$ on $s_6$, we can take $g_6=\mu_2\mu_1\mu_2\mu_1\mu_2$, and then
\[w_g^{s_1}\cdot s_6=s_6\cdot(\mu_2\mu_1\mu_2\mu_1\mu_2)(\mu_1\mu_2)(\mu_2\mu_1\mu_2\mu_1\mu_2)=s_6\cdot\mu_2\mu_1=s_8.\]
We could also take $g_6=\mu_1\mu_2\mu_1\mu_2\mu_1$, and compute
\[w_g^{s_1}\cdot s_6=s_6\cdot(\mu_1\mu_2\mu_1\mu_2\mu_1)(\mu_1\mu_2)(\mu_1\mu_2\mu_1\mu_2\mu_1)=s_6\cdot \mu_2\mu_1=s_8,\]
so the two choices give the same end result, as predicted. It can be checked that in this case $w_g^{s_1}$ is a rotation generating $\dcW$.
\end{eg}

Recall that if $\sim$ is homogeneous, the set $X/{\sim}$ of equivalence classes admits a natural $G$-action by $[x]\cdot g=[x\cdot g]$. Moreover, if $\varphi\in\norm{\Aut_G(X)}{W}$, then $[x]=[y]$ implies $[\varphi(x)]=[\varphi(y)]$, so we have an induced map $\widetilde{\varphi}\in\Aut_G(X/{\sim})$ given by $\widetilde{\varphi}([x])=[\varphi(x)]$.

\begin{prop}
\label{autXmodWnormal}
If $\sim$ is a regular equivalence relation $X$ and $W$ is the corresponding subgroup of $A=\Aut_G(X)$, then the map $\norm{A}{W}\to\Aut_G(X/{\sim})$ given by $\varphi \mapsto\widetilde{\varphi}$ is a homomorphism with kernel $W$. Thus there is an injection $\norm{A}{W}/W\incl\Aut_G(X/{\sim})$. In particular, if $\sim$ is normal, then there is an injection $A/W\incl\Aut_G(X/{\sim})$.
\end{prop}
\begin{proof}
First note that for $\varphi_1,\varphi_2\in\norm{A}{W}$, we have
\[\widetilde{\varphi_2\circ\varphi_1}([x])=[\varphi_2(\varphi_1(x))]=\widetilde{\varphi_2}\circ\widetilde{\varphi_1}([x])\]
for any $x\in X$, so $\varphi\mapsto\widetilde{\varphi}$ is a group homomorphism. If $w\in W$, then $[w(x)]=[x]$ for all $x$, so $\widetilde{w}$ is the identity. If $[\varphi(x)]=[x]$ for some $x$, then there exists $w\in W$ such that $w(x)=\varphi(x)$. Then
\[w(x\cdot g)=w(x)\cdot g=\varphi(x)\cdot g=\varphi(x\cdot g)\]
for all $g\in G$. As $G$ acts transitively on $X$, it follows that $\varphi=w$. Hence the kernel of the map $\varphi\mapsto\widetilde{\varphi}$ is exactly $W$, and so this map induces an injection $\norm{A}{W}/W\incl\Aut_G(X/{\sim})$. The statement when $\sim$ is normal then follows immediately from \Cref{regequivlem}, as $W$ is normal in this case.
\end{proof}
\begin{rem}
The map from \Cref{autXmodWnormal} is not an isomorphism in general. Indeed, let $X$ be the set of vertices of the square
\[\begin{tikzpicture}[scale=1.3]
\node at (0,0) (a) {$a$};
\node at (1,0) (b) {$b$};
\node at (1,-1) (c) {$c$};
\node at (0,-1) (d) {$d$};
\draw (a) -- (b) -- (c) -- (d) -- (a);
\end{tikzpicture}\]
and let $G=D_4$ act on the right. Then
\[\Stab_G(a)=\Stab_G(c)\ne\Stab_G(b)=\Stab_G(d),\]
so $A=\Aut_G(X)$ has order $2$ and is generated by a rotation by $\pi$. In the quotient $A\lquot X$, we have
\[\Stab_G(\{a,c\})=\Stab_G(\{b,d\})\]
so there is an automorphism of $A\lquot X$ that does not lift to an automorphism of $X$.
\end{rem}

\section{The orbit groupoid}
\label{orbitgroupoid}
We will mostly continue to work in the generality of \Cref{regequivrels}, so we have a homogeneous space $X$ for a group $G$, acting on the right, and a subgroup $W$ of $\Aut_G(X)$ acting on the left. However, we will not need the $G$-action at first, so we let $X$ be any set with a left $W$-action. We will form an orbit groupoid $W\lquot X$.

The objects of the groupoid $W\lquot X$ are the $W$-orbits; we write $[x]$ for the orbit of $x$ under $W$. The morphisms $\Hom_{W\lquot X}([x],[y])\eqdef\Iso_W([x],[y])$ are bijections between $W$-orbits that commute with the $W$-action. In other words, a morphism $\varphi\in\Iso_{W}([x],[y])$ makes the diagram
\[\begin{tikzcd}
{[x]}\arrow{r}{\varphi}\arrow{d}[swap]{w\cdot}&{[y]}\arrow{d}{w\cdot}\\
{[x]}\arrow{r}{\varphi}&{[y]}
\end{tikzcd}\]
commute for any $w\in W$. For example, in the case that $X$ carries a right action by $G$, and $W$ is a subgroup of $\Aut_G(X)$, any $g$ with $x\cdot g\in[y]$ induces such a bijection.

Composition of morphisms (read left-to-right) is given by $\varphi*\psi=\psi\circ\varphi$, which is in particular associative. Note that the identity on orbits commutes with the $W$-action, and the composition of two bijections commuting with the $W$-action also commutes with this action. Finally, if $\varphi\colon[x]\to[y]$ is a bijection commuting with the $W$-action, then $\varphi^{-1}\colon[y]\to[x]$ also commutes with the $W$-action. So every morphism is invertible, and the category $W\lquot X$ is indeed a groupoid.

The reason for the choice of composition law is that the morphisms of $W\lquot X$ commute with the left action of $W$, so should act on the right. This is similar to our earlier convention that the automorphisms $\Aut_G(X)$ commuting with a right $G$-action should act on the left. For consistency, composition denoted by $\circ$ is always read right-to-left, and composition denoted by $*$ is read left-to-right.

\begin{rem}
\label{othergroupoids}
Let $T(X)$ be the trivial groupoid on $X$, with exactly one morphism $f_{x,y}\colon x\to y$ for each $x,y\in X$. Then $W$ also acts on $T(X)$, with the action on morphisms given by $w\cdot f_{x,y}=f_{w\cdot x,w\cdot y}$. If $W$ acts freely on $X$, then the orbit groupoid $W\lquot X$ defined above is isomorphic to the orbit groupoid $T(X)\orbgpd W$ defined in \cite{browntopology}*{11.2.1}. The groupoid $T(X)\orbgpd W$ is in fact defined for any $W$-action on $X$, but when the $W$-action is free, it has the additional property of admitting a covering morphism (see \cite{browntopology}*{10.2}) from the simply connected groupoid $T(X)$. If the action is not free, then $W\lquot X$ need not agree with $T(X)\orbgpd W$.

When the $W$-action is free, the orbit groupoid $W\lquot X$ is also isomorphic to the groupoid defined in \cite{krammergarside}*{Defn.~2.4} from the data of a group acting freely on a set. The morphisms in this groupoid are formally given by orbits of the induced action of $W$ on $X\times X$, and as the action is free these orbits are the graphs of the morphisms in $W\lquot X$.
\end{rem}

Two orbits $[x]$ and $[y]$ in $W\lquot X$ are in the same connected component if and only if $\Stab_W(x)$ and $\Stab_W(y)$ are conjugate. In particular, if the $W$-action is free then $W\lquot X$ is connected.

Given any groupoid $\cG$ (with composition read left-to-right), we say that a collection of subgroups $H=\set{H_x}_{x\in\cG}$ of each point group $\cG(x):=\Hom_\cG(x,x)$ is a \emph{normal subgroup of $\cG$} if for every morphism $\alpha\colon x\to y$, we have $\alpha H_y=H_x\alpha$. Given such a collection, it follows that we may define a quotient groupoid $\cG/H$ with the same objects as $\cG$, and morphisms $(\cG/H)(x,y)=H_x\lquot\cG(x,y)=\cG(x,y)/H_y$. If $\Phi\colon\cG_1\to\cG_2$ is a functor, let
\[(\ker{\Phi})_x=\set{f\in\cG_1(x):\Phi(f)=1_{\Phi(x)}}.\]
Then $\ker{\Phi}=\set{(\ker{\Phi})_x}$ is a normal subgroup of $\cG_1$. Proofs of these statements can be found in \cite{browntopology}*{\S8.3}, as can the following lemma.
\begin{lem}[{\cite{browntopology}*{8.3.2}}]
\label{1it}
If a functor $\Phi\colon\cG_1\to\cG_2$ is injective on objects, then it induces an isomorphism of groupoids $\cG_1/\ker{\Phi}\isoto\im{\Phi}$.
\end{lem}

From now on, we assume that $X$ is a homogeneous space for a right $G$-action, and $W\leq\Aut_G(X)$, so that $W$ acts freely. Thus by \Cref{othergroupoids} we can think of the trivial groupoid $T(X)$ as a universal cover for the groupoid $W\lquot X\isom T(X)\orbgpd W$. In this more specific situation, we can identify the morphisms of $W\lquot X$ with certain classes of elements of $G$, as we now explain.

Let $\normgpd{W}{G}$ be the groupoid with objects given by the $W$-orbits $[x]$ in $X$, and morphisms
\[\normgpd{W}{G}([x],[y])=\set{g\in G:x\cdot g\in[y]},\]
with composition $g*h=gh$. Note that the point groups are the groups $\norm{G}{[x]}$ discussed in \Cref{regequivrels}. Now for each $g\in G$, let $\act{g}{}\colon X\to X$ be the function $\act{g}{}(x)=x\cdot g$, and let $\act{g}{x}=\rest{\act{g}{}}{[x]}\colon[x]\to[x\cdot g]$. Each $\act{g}{x}$ is a bijection of orbits commuting with the $W$-action, and we have $\act{g}{x}*\act{h}{x\cdot g}=\act{gh}{x}$, so we may define $\Phi\colon\normgpd{W}{G}{}\to W\lquot X$ to be the functor given by the identity on objects, and by $\Phi(g)=\act{g}{x}$ on morphisms.
\begin{prop}
\label{Gbijprop}
Let $X$ be a homogeneous space for a right $G$-action, and let $W\leq\Aut_G(X)$. Then every bijection $\varphi\colon[x]\to[y]$ of $W$-orbits commuting with the action of $W$ has the form $\act{g}{x}$ for some $g\in G$ with $x\cdot g\in[y]$. Thus $\Phi$ is a surjective functor.
\end{prop}
\begin{proof}
Let $\varphi\colon[x]\to[y]$ commute with the $W$-action. As $G$ acts transitively, there exists $g\in G$ with $x\cdot g=\varphi(x)$. Note that all elements of $[x]$ have the form $w\cdot x$ for some $w\in W$. Then
\[\varphi(w\cdot x)=w\cdot\varphi(x)=w\cdot x\cdot g=\act{g}{x}(w\cdot x),\]
so $\varphi=\act{g}{x}$.
\end{proof}

For each object $[x]$ of $\normgpd{W}{G}{}$, write $(\Stab_G)_{[x]}=\Stab_G(x)\leq\normgpd{W}{G}([x],[x])$. Denote the collection $\set{(\Stab_G)_{[x]}}$ of these subgroups by $\Stab_G$.

\begin{cor}
\label{Gorbitgroupoid}
Let $X$ be a homogeneous space for a right $G$-action, and let $W\leq\Aut_G(X)$. Then $\Stab_G$ is a normal subgroup of $\normgpd{G}{W}$, and the orbit groupoid $W\lquot X$ is isomorphic to the quotient groupoid $\normgpd{G}{W}/\Stab_G$.
\end{cor}
\begin{proof}
The functor $\Phi$ is a bijection on objects, so it follows from \Cref{Gbijprop} and \Cref{1it} that $W\lquot X\isom \normgpd{G}{W}/\ker{\Phi}$. The function $\act{g}{x}$ is the identity if and only if $g\in\Stab_G(x)$, and so $\ker{\Phi}=\Stab_G$.
\end{proof}

\section{The cluster modular groupoid}
\label{clustmodgpd}
Recall from \Cref{labseedgraph} that we have maps $\maut{g}{Q}\in\Aut(\cK)$ for each quiver $Q$ and $g\in M_n$.

\begin{defn}
For $\cS$ a mutation class of labelled seeds, the \emph{cluster modular groupoid} $\CMG(\cS)$ is the groupoid with objects given by the quivers occurring in seeds of $\cS$, and morphisms $\Hom(Q_1,Q_2)$ given by formal symbols $\morph{g}{Q_1}$ for $g\in M_n$ such that $Q_1\cdot g=Q_2$, satisfying the relations
\[\morph{g}{Q_1}=\morph{h}{Q_1}\iff\maut{g}{Q_1}=\maut{h}{Q_1}.\]
The (left-to-right) composition is $\morph{g}{Q}*\morph{h}{Q\cdot g}=\morph{gh}{Q}$,
which is well-defined because $\maut{g}{Q}\circ\maut{h}{Q\cdot g}=\maut{gh}{Q}$.
\end{defn}
Our groupoid $\CMG(\cS)$ is a full subgroupoid of the cluster modular groupoid defined by Fock--Goncharov \cite{fockcluster1};
the only difference is that we have a fixed vertex set for the quivers, rather than allowing all possible $n$ element sets.
As Fock--Goncharov's groupoid is connected, ours is equivalent to it.
Note that $\CMG(\cS)$ does not depend on the field $\cF$, but only on the mutation class of quivers underlying the mutation class $\cS$ of seeds.
 
We now consider two particular instances of the orbit groupoid construction outlined in \Cref{orbitgroupoid}. 
Let $\cS$ be a cluster structure;
for each $x\in\cS$ we write $x=(Q_x,\beta_x)$. 
By \Cref{regequivlem}, there exist subgroups $\cW$ and $\dcW$ of $\Aut_{M_n}(\cS)$ corresponding to the two equivalence relations from \Cref{regequivdefs}, so that $\cW$-orbits are precisely similarity classes, i.e.\ equivalence classes under $\approx$,
and $\dcW$-orbits are equivalence classes under $\simeq$. The symbols $\cW$ and $\dcW$ will denote these specific groups for the remainder of the paper.

We may think of the group $\Aut{\cK}$ as a groupoid with the single object $\cK$. There is a functor $\normgpd{\dcW}{G}\to\Aut{\cK}$ mapping every $\dcW$-orbit to $\cK$, and each morphism $g\colon[x]\to[y]$ to $\maut{g}{Q_x}$. This functor is well-defined, as $Q_x=Q_{w\cdot x}$ for all $w\in \dcW$, and because $\maut{g}{Q_x}\circ\maut{h}{Q_y}=\maut{gh}{Q_x}$ for any $g\colon[x]\to[y]$ and $h\colon[y]\to[z]$.

Note that if $Q_1\approx Q_2$, i.e.\ $Q_1$ and $Q_2$ differ by reversing the orientation of some components, then $\maut{g}{Q_1}=\maut{g}{Q_2}$ for any $g\in M_n$. Therefore the functor $\normgpd{\cW}{G}\to\Aut{\cK}$ mapping every $\cW$-orbit to $\cK$ and every $g\colon[x]\to[y]$ to $\maut{g}{Q_x}$ is well-defined, as $Q_x\approx Q_{w\cdot x}$ for all $w\in \cW$.

Because $\maut{g}{Q_x}=\id{\cK}$ for $g\in\Stab_G(x)$, these functors define left actions of the groupoids $\normgpd{\cW}{G}/\Stab_G$ and $\normgpd{\dcW}{G}/\Stab_G$ on the field $\cK$. Each morphism $g\colon[x]\to[x\cdot g]$ acts as the change of coordinates of $\cF$ from $\beta_x$ to $\beta_{x\cdot g}$. In other words, the diagram

\[\begin{tikzcd}[column sep=large,row sep=large]
\cK\arrow{r}{\beta_x}&\cF\\
\cK\arrow{u}{\maut{g}{Q_x}}\arrow{ur}[swap]{\beta_{x\cdot g}}
\end{tikzcd}\]
commutes for all $x\in\cS$ and $g\in M_n$.

\begin{thm}
\label{cmgthm}
The cluster modular groupoid $\CMG(\cS)$ and the orbit groupoid $\dcW\lquot\cS$ are isomorphic.
\end{thm}
\begin{proof}
We will deduce the result by constructing a surjective functor $\Psi\colon\normgpd{\dcW}{G}\to\CMG(\cS)$, bijective on objects, and observing that $\ker{\Psi}=\ker{\Phi}$ for $\Phi\colon\normgpd{\dcW}{G}\to \dcW\lquot\cS$ the functor from \Cref{orbitgroupoid}, thus obtaining isomorphisms
\[\CMG(\cS)\isom\normgpd{\dcW}{G}/\ker{\Psi}=\normgpd{\dcW}{G}/\ker{\Phi}\isom \dcW\lquot\cS,\]
with the last isomorphism coming from \Cref{Gorbitgroupoid}. We define $\Psi$ on objects by taking each $\dcW$-orbit $[x]$ to its common quiver $Q_x$, and on morphisms by $\Psi(g)=\morph{g}{Q_x}$, for $g\in \normgpd{\dcW}{G}([x],[y])$. 
This is well-defined as for $g\colon[x]\to[y]$ and $h\colon[y]\to[z]$, we have
\[\Psi(g*h)=\Psi(gh)=\morph{gh}{Q_x}=\morph{g}{Q_x}*\morph{h}{Q_y}=\Psi(g)*\Psi(h).\]
By the definition of $\CMG(\cS)$, this functor is bijective on objects. It is surjective on morphisms as if $\morph{g}{Q_1}\colon Q_1\to Q_2$ is a morphism in $\CMG(\cS)$, then there exists $x\in\cS$ with $Q_x=Q_1$, and $Q_{x\cdot g}=Q_2$.
Thus $g\in\normgpd{\dcW}{G}([x],[x\cdot g])$ and $\morph{g}{Q_1}=\Psi(g)$.

It remains to show that $(\ker{\Psi})_{[x]}=(\ker{\Phi})_{[x]}=\Stab_G(x)$. Let $g\in\normgpd{\dcW}{G}([x])$, and note that $\Psi(g)=\morph{g}{Q_x}=\morph{1}{Q_x}=1_{Q_x}$ if and only if $\maut{g}{Q_x}=\maut{1}{Q_x}=\id{\cK}$. Now if $g\in\Stab_G(x)$, then $\maut{g}{Q_x}=\beta_x^{-1}\circ\beta_x=1_{\cK}$, so $\Psi(g)=\id{Q_x}$. Conversely, if $\Psi(g)=\id{Q_x}$ then $\beta_{x\cdot g}=\beta_x$, and $Q_{x\cdot g}=Q_x$ because $g\in\normgpd{\dcW}{G}([x])$, so $g\in\Stab_G(x)$.

We conclude that $\ker{\Psi}=\ker{\Phi}$, and thus obtain the required isomorphism.
\end{proof}

\section{The cluster automorphism group}
\label{clustautgp}
Let $\cS$ be a cluster structure on $\cF$.
We define a left action of $\cW\leq\Aut_{M_n}(\cS)$ on $\cF$ via the map $\alpha\colon\cW\to\Aut(\cF)$ given by
\[\alpha\colon w\mapsto\waut{w}=\beta_{w\cdot x}\circ\beta_x^{-1}\]
for $x\in\cS$ any labelled seed. If $y$ is another labelled seed, then there exists $g\in M_n$ with $x\cdot g=y$, so $\beta_y=\beta_{x\cdot g}=\beta_x\circ\maut{g}{Q_x}$. It follows that
\begin{align*}
\beta_{w\cdot y}\circ\beta_y^{-1}&=\beta_{w\cdot x\cdot g}\circ(\maut{g}{Q_x})^{-1}\circ\beta_x^{-1}\\
&=\beta_{w\cdot x}\circ\maut{g}{Q_{w\cdot x}}\circ(\maut{g}{Q_x})^{-1}\circ\beta_x^{-1}\\
&=\beta_{w\cdot x}\circ\beta_x^{-1},
\end{align*}
as $Q_{w\cdot x}\approx Q_x$, so $\maut{g}{Q_x}=\maut{g}{Q_{w\cdot x}}$. Thus the definition of $\waut{w}$ is independent of the choice of seed $x$. As $\dcW$ is a subgroup of $\cW$, this action restricts to an action of $\dcW$ on $\cF$. We write $\alpha^+\defeq\rest{\alpha}{\dcW}$.

To see that we have defined a left action, let $v,w\in\cW$, and let $x\in\cS$. We can write $\waut{w}=\beta_{w\cdot x}\circ\beta_x^{-1}$, and $\waut{v}=\beta_{v\cdot(w\cdot x)}\circ\beta_{w\cdot x}^{-1}$. So
\[\waut{v}\circ\waut{w}=\beta_{v\cdot(w\cdot x)}\circ\beta_{w\cdot x}^{-1}\circ\beta_{w\cdot x}\circ\beta_x^{-1}=\beta_{vw\cdot x}\circ\beta_x^{-1}=\waut{vw}.\]

As the action of $M_n$ is transitive, for any $x\in\cS$ and $w\in\cW$ there exists $g\in M_n$ such that $x\cdot g=w\cdot x$. Then the diagram
\[\begin{tikzcd}[column sep=5em,row sep=5em]
\cK\arrow{r}{\beta_x}&\cF\\
\cK\arrow{u}{\maut{g}{Q_x}}\arrow{ur}[swap]{\beta_{w\cdot x}}\arrow{r}[swap]{\beta_x}&\cF\arrow{u}[swap]{\waut{w}}
\end{tikzcd}\]
commutes. The map $\maut{g}{Q_x}$ on the left is the action of $g\Stab_G(x)\in\Aut_{\normgpd{\cW}{G}/\Stab_G}([x])$ on $\cK$, and the map $\waut{w}$ on the right is the action of $w$ on $\cF$. Thus each isomorphism $\beta_x$ intertwines the action of the groupoid $\normgpd{\cW}{G}/\Stab_G$ on $\cK$ with the action of the group $\cW$ on $\cF$, and similarly for $\normgpd{\dcW}{G}/\Stab_G$ and $\dcW$. We can use the isomorphisms of \Cref{Gorbitgroupoid} and \Cref{cmgthm} to make equivalent statements for the groupoids $\cW\lquot\cS$, $\dcW\lquot\cS$ and $\CMG(\cS)$.

These actions of $\cW$ and $\dcW$ on $\cF$ restrict to $\QQ$-algebra automorphisms of the cluster algebra $\cA$, as they are automorphisms of $\cF$ that permute the set of cluster variables. 
For any $w\in\cW$ and any seed $x$, the action of $w$ on $\cF$ sends the labelled cluster corresponding to $x$ to that corresponding to $w\cdot x$, and commutes with the action of each $\mu_i$ as $x$ and $w\cdot x$ have similar quivers. 
Thus this action is a cluster automorphism in the sense of \cite{assemcluster}*{Defn.~1}. If $w\in\dcW\leq\cW$, then the corresponding cluster automorphism will be direct, as it sends clusters to clusters with the same quiver. 
This observation provides maps $\alpha\colon\cW\to\cAut{\cA}$ and $\alpha^+\colon\dcW\to\dircAut{\cA}$, recalling that $\cAut{\cA}$ is the group of cluster automorphisms, and $\dircAut{\cA}$ is the subgroup of direct cluster automorphisms.

\begin{thm}
\label{clustautthm}
The maps $\alpha\colon\cW\to\cAut{\cA}$ and $\alpha^+\colon\dcW\to\dircAut{\cA}$ are group isomorphisms.
\end{thm}
\begin{proof}
The groups $\cW$ and $\dcW$ both act faithfully on $\cF$, else there are two labelled seeds with the same labelled cluster, contradicting \cite{gekhtmanproperties}*{Thm.~4}. Thus the maps $\alpha$ and $\alpha^+$ are injective.

Let $(Q,\beta)\in\cS$, and let $\alpha$ be a cluster automorphism, so there exists $(Q',\gamma)\in\cS$ such that $\alpha(\beta_{\sigma(i)})=\gamma_i$ for some permutation $\sigma\colon I\to I$. By a small modification to \cite{assemcluster}*{Lem.~2.3}, to allow for disconnected quivers, the quiver $Q^\sigma$ lies in the similarity class of $Q'$. We have $\sigma\in\Sym{n}\leq M_n$, so the seed $(Q^\sigma,\beta^\sigma)$ is mutation equivalent to $(Q,\beta)$.

As $Q^\sigma\approx Q'$, there exists $w\in\cW$ such that $w\cdot(Q^\sigma,\beta^\sigma)=(Q',\gamma)$, and so $\waut{w}=\gamma\circ(\beta^\sigma)^{-1}$. For any $i$ we have
\[\waut{w}(\beta_{\sigma(i)})=\waut{w}(\beta^\sigma_i)=\gamma_i,\]
so the action of $w$ agrees with that of $\alpha$ on each $\beta_i$, and the $\beta_i$ form a free generating set of $\cF$, so $\alpha=\waut{w}$.

If $f$ is taken to be a direct cluster automorphism, then $Q^\sigma=Q'$, and so we can take $w\in \dcW$.
\end{proof}

The conclusion of \Cref{clustautthm} would not hold if $M_n$ was replaced by the free group on $n$ involutions. As noted above, the maps $\waut{w}$ take clusters to clusters and commute with mutations, so they are cluster automorphisms, but they also satisfy the stronger property of taking labelled clusters to labelled clusters. This means that if $(\beta_1,\dotsc,\beta_n)$ is some labelled cluster of $\cA$, then $(f(\beta_1),\dotsc,f(\beta_n))$ is a labelled cluster of $\cA$. A priori, cluster automorphisms need only take labelled clusters to permutations of labelled clusters, but the presence of permutations in $M_n$ means that any permutation of a labelled cluster from a mutation class is also a labelled cluster from that mutation class, so in fact the stronger property also holds.

The isomorphism classes of the groups $\cAut{\cA}$ and $\dircAut{\cA}$ for all Dynkin and affine types are shown in \cite{assemcluster}*{Table~1}, and so this table provides isomorphism classes for the groups $\cW$ and $\dcW$ for these types.

\begin{cor}
\label{clustautisCMG}
The group $\dircAut{\cA}$ of direct cluster automorphisms of a cluster algebra $\cA$ is isomorphic to each point group $\Aut_{\CMG(\cS)}(Q)$ in the cluster modular groupoid $\CMG(\cS)$, for $\cS$ the set of labelled seeds of $\cA$.
\end{cor}
\begin{proof}
By \Cref{cmgthm}, $\CMG(\cS)$ is isomorphic to $\dcW\lquot\cS$, so if $x$ is any labelled seed of $\cA$ with quiver $Q$, then
\[\Aut_{\CMG(\cS)}(Q)\isom\Aut_{\dcW\lquot\cS}([x])\isom\norm{M_n}{[x]}/\Stab_{M_n}(x)\isom \dcW\isom\dircAut{\cA},\]
with the final three isomorphisms provided by \Cref{Gorbitgroupoid}, \Cref{Wissubquotient} and \Cref{clustautthm} respectively.
\end{proof}

Another way to obtain an isomorphism $\dircAut{\cA}\isom\Aut_{\CMG(\cS)}(Q)$ is as follows. Let \[B_Q=\set{\beta_x:x\in\cS,Q_x=Q}.\]
Then $\Aut_{\CMG(\cS)}(Q)$ acts freely and transitively on the right of $B_Q$ by
\[\beta_x\cdot\morph{g}{Q}=\beta_x\circ\maut{g}{Q}=\beta_{x\cdot g}\]
and $\dircAut{\cA}$ acts freely and transitively on the left of $B_Q$ by $\alpha\cdot\beta_x=\alpha\circ\beta_x$. Therefore each choice of seed $x$ with quiver $Q$ provides an isomorphism $\dircAut{\cA}\isom\Aut_{\CMG(\cS)}(Q)$ in the same manner as in the proof of \Cref{Wissubquotient}. Furthermore, this isomorphism agrees with that of \Cref{clustautisCMG} provided the same seed $x$ is chosen to obtain the isomorphism $\norm{M_n}{[x]}/\Stab_{M_n}(x)\isom \dcW$ of \Cref{Wissubquotient}.

\section{Quivers determined by stabilisers}
\label{stabsdeterminequivs}
\begin{defn}
A mutation class $\cS$ is \emph{small} if only finitely many quivers occur among its labelled seeds.
\end{defn}
Thus a mutation class is small if and only if one of its seeds has a quiver of finite mutation type, or equivalently if all of them do. We prefer the term `small' to `finite mutation type' when referring to the class of seeds rather than to one of its quivers.

As we have a classification of quivers of finite mutation type, we also have a classification of small mutation classes. 
Precisely, they are classes in which either the quivers have two vertices, or no quiver has an arrow of multiplicity more than $2$; see \cite{derksennewgraphs}*{Cor.~8}. 
All finite mutation classes are small, as are all mutation classes arising from tagged triangulations of marked bordered surfaces, as described in \cite{fomintriangulated1}. 
Mutation classes in which the total multiplicity of arrows in the quiver is constant across all seeds, classified in \cite{ladkaniwhich}, are necessarily also small.

\begin{thm}
\label{mainthm}
Let $\cS$ be a small mutation class and let $(Q_1,\beta_1),(Q_2,\beta_2)\in \cS$ be labelled seeds. 
Then $\Stab_{M_n}(Q_1,\beta_1)=\Stab_{M_n}(Q_2,\beta_2)$ if and only if $Q_1$ and $Q_2$ are similar.
\end{thm}

\begin{proof}
As explained in \Cref{regequivdefs}, the equivalence relation $\approx$ is regular, so labelled seeds with similar quivers have the same stabiliser under the $M_n$-action. It remains to prove the converse.

First assume that $Q_1$ and $Q_2$ have two vertices. In this case, there is nothing to prove; no mutation of $Q_1$ can alter the multiplicity of the arrow between its two vertices, so $Q_2$ has an arrow of the same multiplicity, and $Q_1$ and $Q_2$ are similar.

From now on, we assume that $Q_1$ and $Q_2$ have more than two vertices, and that $\Stab_{M_n}(Q_1,\beta_1)=\Stab_{M_n}(Q_2,\beta_2)$. We make the following claims, the proofs of which are deferred to \Cref{edgecountlem,edgeorientlem}:
\begin{itemize}
\item[(a)]if $(Q,\beta)$ is a labelled seed in a small mutation class, and $Q$ has at least $3$ vertices, then the underlying weighted graph of $Q$ is determined by $\Stab_{M_n}(Q,\beta)$,
\item[(b)]if $(Q,\beta)$ is a labelled seed in a small mutation class, and $Q$ has at least $3$ vertices, then the relative orientation of any pair of adjacent arrows in $Q$, i.e.\ whether or not they form a directed path, is determined by $\Stab_{M_n}(Q,\beta)$.
\end{itemize}

Now by (a), any two vertices $i$ and $j$ have an arrow of the same multiplicity between them in both $Q_1$ and $Q_2$. Thus we can treat $Q_1$ and $Q_2$ as being two orientations of the same weighted graph $\Gamma$.

Let $e$ be an edge of $\Gamma$. If $f$ is any edge adjacent to $e$, then by (b) its orientation relative to $e$ is determined by the stabiliser, and so this relative orientation is the same in $Q_1$ and $Q_2$. It follows that $f$ has the same orientation in $Q_1$ and $Q_2$ if and only if $e$ does. The same now applies to any edge adjacent to $f$, and so on, and we deduce that the components of $e$ in $Q_1$ and $Q_2$ are either the same or opposite.

Thus $Q_1$ and $Q_2$ can only differ by taking the opposite of some collection of connected components, and so $Q_1$ and $Q_2$ are similar, as required.
\end{proof}

We will see later in \Cref{evilquiver} that the result of \Cref{mainthm} may fail for mutation classes that are not small.

\begin{cor}
\label{autdeterminesquiv}
If $\cS$ is a small mutation class, then $\cW=\Aut_{M_n}(\cS)$, and so the $\Aut_{M_n}(\cS)$-orbits in $\cS$ are precisely the sets of labelled seeds with similar quivers.
\end{cor}

\begin{proof}
By \Cref{mainthm}, two quivers have the same stabiliser if and only if their quivers are similar. Thus the equivalence relation $\approx$, corresponding to the subgroup $\cW$ of $\Aut_{M_n}(\cS)$, is the same as the relation of having the same stabiliser, which corresponds to the entire automorphism group. Thus $\cW=\Aut_{M_n}(\cS)$. So $\Aut_{M_n}(\cS)$-orbits are $\cW$-orbits, which are the sets of labelled seeds with similar quivers.
\end{proof}

Combining \Cref{autdeterminesquiv} with \Cref{clustautthm}, we immediately obtain the following result.

\begin{cor}
If $\cS$ is a small mutation class, then $\Aut_{M_n}(\cS)\isom\cAut{\cA}$, an explicit isomorphism being given by the action of $\Aut_{M_n}(\cS)$ on $\cF$ as described in \Cref{clustautgp}.
\end{cor}

For Dynkin and affine types, which are all small, the isomorphism classes of $\cAut{\cA}$, and hence of $\Aut_{M_n}(\cS)$, can be read off from \cite{assemcluster}*{Table~1}.

An interesting question, to which we do not know the answer, is the following.

\begin{prob}
Is the converse to \Cref{autdeterminesquiv} is also true? In other words, does the property $\cW=\Aut_{M_n}(\cS)$ characterise small mutation classes?
\end{prob}

\begin{cor}
\label{smallnormal}
If $\cS$ is a small mutation class of labelled seeds with connected quivers, then the equivalence relations $\simeq$ and $\approx$ are normal.
\end{cor}
\begin{proof}
Recall from \Cref{regequivlem} that $\simeq$ and $\approx$ are normal if and only if the corresponding groups $\cW$ and $\dcW$ are normal in $\Aut_{M_n}(\cS)$. If $\cS$ is small, then $\cW=\Aut_{M_n}(\cS)$, so is normal. As the quivers of labelled seeds of $\cS$ are connected, we may use \cite{assemcluster}*{Thm.~2.11} in conjunction with \Cref{clustautthm} to see that $\dcW$ has index $1$ or $2$ in $\cW=\Aut_{M_n}(\cS)$, so is also normal.
\end{proof}

It remains to prove the claims (a) and (b) made in the proof of \Cref{mainthm}. To achieve this, we will consider cluster algebras from ice quivers, in order to employ a similar `principal coefficient trick' to that used in the proof of \cite{cerullilinear}*{Cor.~5.3}. An ice quiver is a quiver $Q$ with a partition of its vertices into \emph{mutable vertices} and \emph{frozen vertices}. To remain consistent with earlier notation, our ice quivers will have vertex set $I=\set{1,\dotsc,n}$, partitioned into a set $J$ of mutable vertices, and a set $F$ of frozen vertices. We may define labelled seeds $(Q,\beta)$ with $Q$ an ice quiver, and $\beta=(\beta_1,\dotsc,\beta_n)\in\cF^n$ as before. We will only allow mutations at mutable vertices, and permutations of the labels of the mutable vertices. Thus the mutation class of a labelled seed $(Q,\beta)$, where $Q$ is an ice quiver, is the orbit $(Q,\beta)\cdot M_J$, where $M_J\leq M_n$ is the subgroup generated by $\mu_j$ for $j\in J$, and $\sigma\in\Sym{n}$ such that $\sigma$ fixes $F$ pointwise. Consequently, if $(Q',\beta')$ is any labelled seed in this mutation class, we have $\beta'_k=\beta_k$ for $k\in F$, and
\[\beta'_i\in\QQ(\beta_j:j\in J)[\beta_k:k\in F]\subseteq\cF.\]

Let $Q$ be an ice quiver. Given a labelled seed $(Q,\beta)$, we define $(Q_\circ,\beta_\circ)$ to be the corresponding labelled seed with \emph{trivial coefficients}, so $Q_\circ$ is the full subquiver on the mutable vertices of $Q$ and $\beta_\circ=(\beta_j)_{j\in J}$. If $g\in M_J$ satisfies $(Q,\beta)\cdot g=(Q',\beta')$, then $(Q_\circ,\beta_\circ)\cdot g=(Q'_\circ,\beta'_\circ)$, as $Q'_\circ$ is the full subquiver on the mutable vertices of $Q'$, and $(\beta'_\circ)_j$ is the image of $\beta'_j$ under the projection $\QQ(\beta_j:j\in J)[\beta_k:k\in F]\to\QQ(\beta_j:j\in J)$ defined by
\[\beta_i\mapsto\begin{cases}\beta_i,&i\in J\\1,&i\in F\end{cases}\]
We will use the contrapositive of this result; if $(Q_\circ,\beta_\circ)\cdot g\ne(Q_\circ,\beta_\circ)$, then $(Q,\beta)\cdot g\ne(Q,\beta)$.

Similarly, given $(Q,\beta)$, we define $(Q_\bullet,\beta_\bullet)$ to be the corresponding labelled seed with \emph{principal coefficients}. The quiver $Q_\bullet$ has vertex set $J\sqcup J'$, where $J'=\set{j':j\in J}$ is a clone of the set $J$. The full subquiver of $Q_\bullet$ on $J$ agrees with that of $Q$ on $J$, i.e.\ $Q_\circ$, and there are additional arrows $j'\to j$ for all $j\in J$. Take $\beta_\bullet=(\beta_j,\beta_{j'})_{j\in J}$, where the $\beta_{j'}$ are formal symbols. Then \cite{fomincluster4}*{Thm.~3.7} explains how, for any $g\in M_J$, the labelled cluster of $(Q,\beta)\cdot g$ is determined by that of $(Q_\bullet,\beta_\bullet)\cdot g$. As labelled clusters determine labelled seeds by \cite{gekhtmanproperties}*{Thm.~4}, it follows that the entire seed $(Q,\beta)\cdot g$ is determined by $(Q_\bullet,\beta_\bullet)\cdot g$. In particular, if $(Q_\bullet,\beta_\bullet)\cdot g=(Q_\bullet,\beta_\bullet)$, then $(Q,\beta)\cdot g=(Q,\beta)$. 

\begin{lem}\label{edgecountlem}
Let $i$ and $j$ be two distinct vertices in a quiver $Q$ from a labelled seed $(Q,\beta)\in\cS$. Then:
\begin{itemize}
\item[(i)]if there are no arrows between $i$ and $j$, then $(\mu_i\mu_j)^2\in\Stab_{M_n}(Q,\beta)$,
\item[(ii)]if there is an arrow of multiplicity $1$ between $i$ and $j$, then $(\mu_i\mu_j)^5\in\Stab_{M_n}(Q,\beta)$ and $(\mu_i\mu_j)^2\notin\Stab_{M_n}(Q,\beta)$, and
\item[(iii)]if there is an arrow of multiplicity $2$ or more between $i$ and $j$, then $(\mu_i\mu_j)^N\notin\Stab_{M_n}(Q,\beta)$ for any $N$.
\end{itemize}
Thus, in the case that $\cS$ is small and $Q$ has at least three vertices, so there are no arrows of multiplicity more than $2$, we can determine the multiplicity of the arrow between $i$ and $j$ from the stabiliser of the labelled seed $(Q,\beta)$, proving claim (a).
\end{lem}
\begin{proof}
As all sequences of mutations under consideration only involve mutating at the vertices $i$ and $j$, we may assume that $Q$ is an ice quiver, with $J=\set{i,j}$ and $F=I\setminus J$. We may also assume, by taking opposites if necessary, that any arrow between $i$ and $j$ is oriented towards $j$.

We first consider $(Q_\circ,\beta_\circ)$. If $Q_\circ$ is the quiver
\[\begin{tikzpicture}[scale=1.5]
\node at (0,0) (i) {$i$};
\node at (1,0) (j) {$j$};
\path[-angle 90]
	(i) edge (j);
\end{tikzpicture}\]
of type $\typA2$, then our calculation from \Cref{a2eg} shows that $(Q_\circ,\beta_\circ)$ is not fixed by $(\mu_i\mu_j)^2$, and hence neither is $(Q,\beta)$. Similarly, if $Q_\circ$ is
\[\begin{tikzpicture}[scale=1.5]
\node at (0,0) (i) {$i$};
\node at (1,0) (j) {$j$};
\path[-angle 90,font=\scriptsize]
	(i) edge node [above] {$k$} (j);
\end{tikzpicture}\]
for $k\geq2$, then the seed $(Q_\circ,\beta_\circ)$ has infinite mutation class, by \cite{fomincluster2}*{Thm.~1.4}. If it were fixed by $(\mu_i\mu_j)^N$, the mutation class would have size at most $2N$. Thus $(\mu_i\mu_j)^N$ does not fix $(Q_\circ,\beta_\circ)$ or $(Q,\beta)$.

It remains to show that $(\mu_i\mu_j)^2$ and $(\mu_i\mu_j)^5$ do fix the appropriate labelled seeds. If there are no arrows between $i$ and $j$ in $Q$, then $Q_\bullet$ is
\[\begin{tikzpicture}[scale=1.5]
\node at (0,0) (i') {$\boxed{i'}$};
\node at (1,0) (j') {$\boxed{j'}$};
\node at (0,-1) (i) {$i$};
\node at (1,-1) (j) {$j$};
\path[-angle 90]
	(i') edge (i)
	(j') edge (j);
\end{tikzpicture}\]
where the boxed vertices are frozen. It can be verified that any labelled seed with this quiver is fixed by $(\mu_i\mu_j)^2$, and thus so is $(Q,\beta)$. This is equivalent to showing that any such labelled seed has the same image under $\mu_i\mu_j$ and $\mu_j\mu_i$, which we may check directly; we have
\begin{align*}
\left(\mathord{\begin{tikzpicture}[baseline=0]
\node at (0,0.65) (i') {$\boxed{i'}$};
\node at (1,0.65) (j') {$\boxed{j'}$};
\node at (0,-0.65) (i) {$i$};
\node at (1,-0.65) (j) {$j$};
\path[-angle 90]
	(i') edge (i)
	(j') edge (j);
\end{tikzpicture}},(\beta_i,\beta_j,\beta_{i'},\beta_{j'})\right)\cdot\mu_i\mu_j&=\left(\mathord{\begin{tikzpicture}[baseline=0]
\node at (0,0.65) (i') {$\boxed{i'}$};
\node at (1,0.65) (j') {$\boxed{j'}$};
\node at (0,-0.65) (i) {$i$};
\node at (1,-0.65) (j) {$j$};
\path[-angle 90]
	(i) edge (i')
	(j') edge (j);
\end{tikzpicture}},\left(\frac{\beta_{i'}}{\beta_i},\beta_j,\beta_{i'},\beta_{j'}\right)\right)\cdot\mu_j\\
&=\left(\mathord{\begin{tikzpicture}[baseline=0]
\node at (0,0.65) (i') {$\boxed{i'}$};
\node at (1,0.65) (j') {$\boxed{j'}$};
\node at (0,-0.65) (i) {$i$};
\node at (1,-0.65) (j) {$j$};
\path[-angle 90]
	(i) edge (i')
	(j) edge (j');
\end{tikzpicture}},\left(\frac{\beta_{i'}}{\beta_i},\frac{\beta_{j'}}{\beta_j},\beta_{i'},\beta_{j'}\right)\right)
\end{align*}
and
\begin{align*}
\left(\mathord{\begin{tikzpicture}[baseline=0]
\node at (0,0.65) (i') {$\boxed{i'}$};
\node at (1,0.65) (j') {$\boxed{j'}$};
\node at (0,-0.65) (i) {$i$};
\node at (1,-0.65) (j) {$j$};
\path[-angle 90]
	(i') edge (i)
	(j') edge (j);
\end{tikzpicture}},(\beta_i,\beta_j,\beta_{i'},\beta_{j'})\right)\cdot\mu_j\mu_i&=\left(\mathord{\begin{tikzpicture}[baseline=0]
\node at (0,0.65) (i') {$\boxed{i'}$};
\node at (1,0.65) (j') {$\boxed{j'}$};
\node at (0,-0.65) (i) {$i$};
\node at (1,-0.65) (j) {$j$};
\path[-angle 90]
	(i') edge (i)
	(j) edge (j');
\end{tikzpicture}},\left(\beta_i,\frac{\beta_{j'}}{\beta_j},\beta_{i'},\beta_{j'}\right)\right)\cdot\mu_i\\
&=\left(\mathord{\begin{tikzpicture}[baseline=0]
\node at (0,0.65) (i') {$\boxed{i'}$};
\node at (1,0.65) (j') {$\boxed{j'}$};
\node at (0,-0.65) (i) {$i$};
\node at (1,-0.65) (j) {$j$};
\path[-angle 90]
	(i) edge (i')
	(j) edge (j');
\end{tikzpicture}},\left(\frac{\beta_{i'}}{\beta_i},\frac{\beta_{j'}}{\beta_j},\beta_{i'},\beta_{j'}\right)\right)
\end{align*}
By a similar direct calculation, we may show that if $Q_\bullet$ is
\[\begin{tikzpicture}[scale=1.5]
\node at (0,0) (i'1) {$\boxed{i'}$};
\node at (1,0) (j'1) {$\boxed{j'}$};
\node at (0,-1) (i1) {$i$};
\node at (1,-1) (j1) {$j$};
\node at (1.5,-0.5) (or) {or};
\node at (2,0) (i'2) {$\boxed{i'}$};
\node at (3,0) (j'2) {$\boxed{j'}$};
\node at (2,-1) (i2) {$i$};
\node at (3,-1) (j2) {$j$};
\path[-angle 90]
	(i'1) edge (i1)
	(i1) edge (j1)
	(j'1) edge (j1)
	(i'2) edge (i2)
	(j'2) edge (j2)
	(j2) edge (i2);
\end{tikzpicture}\]
then $(Q_\bullet,\beta_\bullet)$ is fixed by $(\mu_i\mu_j)^5$, and hence so is $(Q,\beta)$.
\end{proof}

\begin{rem}
\label{reducetomult1}
Let $(Q,\beta)$ be a seed in a small mutation class, such that $Q$ has arrows between vertices $i$ and $j$ and between vertices $j$ and $k$. We wish to determine the relative orientation of these two arrows. Let $Q'$ be the full subquiver of $Q$ on the vertices $i$, $j$, $k$.

As $(Q,\beta)$ lies in a small mutation class, the maximal multiplicity of any arrow of $Q'$ is $2$. If $Q'$ has a multiplicity $2$ arrow, then by the classification of $3$-vertex quivers with finite mutation class in \cite{derksennewgraphs}*{Thm.~7}, $Q'$ is a directed $3$-cycle in which either all three arrows have multiplicity $2$, or one arrow has multiplicity $2$ and the others have multiplicity $1$. In either case, $Q'$ must contain a directed path through $j$. Thus, to prove that the stabiliser of $(Q,\beta)$ determines the relative orientation of the two arrows under consideration, it suffices to consider the case that $Q'$ only has arrows of multiplicity $1$.
\end{rem}

\begin{lem}\label{edgeorientlem}
Let $i$, $j$ and $k$ be three distinct vertices in a quiver $Q$ from a labelled seed $(Q,\beta)$, such that there is an arrow between $i$ and $j$ and an arrow between $j$ and $k$. Let $Q'$ be the full subquiver on $i$, $j$ and $k$, and assume all arrows of $Q'$ have multiplicity $1$. Then:
\begin{itemize}
\item[(i)]if there is no arrow between $i$ and $k$, then $(\mu_i\mu_j\mu_k)^6\in\Stab_{M_n}(Q,\beta)$ if and only if $Q'$ is a directed path through $j$, so the stabiliser can distinguish
\[\begin{tikzpicture}[scale=1.3]
\node at (0,0) (i1) {$i$};
\node at (1,0) (j1) {$j$};
\node at (2,0) (k1) {$k$};
\node at (2.5,0) (or) {or};
\node at (3,0) (i2) {$i$};
\node at (4,0) (j2) {$j$};
\node at (5,0) (k2) {$k$};
\path[-angle 90]
	(i1) edge (j1)
	(j1) edge (k1)
	(k2) edge (j2)
	(j2) edge (i2);
\end{tikzpicture}\]
from
\[\begin{tikzpicture}[scale=1.3]
\node at (0,0) (i1) {$i$};
\node at (1,0) (j1) {$j$};
\node at (2,0) (k1) {$k$};
\node at (2.5,0) (or) {or};
\node at (3,0) (i2) {$i$};
\node at (4,0) (j2) {$j$};
\node at (5,0) (k2) {$k$};
\path[-angle 90]
	(i1) edge (j1)
	(k1) edge (j1)
	(j2) edge (k2)
	(j2) edge (i2);
\end{tikzpicture}\]
\item[(ii)]if there is an arrow between $i$ and $k$, then $Q'$ contains a directed path through $j$ if and only if $(\mu_i\mu_k\mu_i\mu_k\mu_i\mu_j)^2\notin\Stab_{M_n}(Q,\beta)$, so the stabiliser can distinguish
\[\begin{tikzpicture}[scale=1.2]
\node at (0,0) (i1) {$i$};
\node at (1,0) (j1) {$j$};
\node at (2,0) (k1) {$k$,};
\node at (2.5,0) (i2) {$i$};
\node at (3.5,0) (j2) {$j$};
\node at (4.5,0) (k2) {$k$,};
\node at (5,0) (i3) {$i$};
\node at (6,0) (j3) {$j$};
\node at (7,0) (k3) {$k$,};
\node at (7.5,0) (or) {or};
\node at (8,0) (i4) {$i$};
\node at (9,0) (j4) {$j$};
\node at (10,0) (k4) {$k$};
\path[-angle 90]
	(i1) edge (j1)
	(j1) edge (k1)
	(k1) edge [bend left] (i1)
	(i2) edge (j2)
	(j2) edge (k2)
	(i2) edge [bend right] (k2)
	(k3) edge (j3)
	(j3) edge (i3)
	(i3) edge [bend right] (k3)
	(k4) edge (j4)
	(j4) edge (i4)
	(k4) edge [bend left] (i4);
\end{tikzpicture}\]
from
\[\begin{tikzpicture}[scale=1.2]
\node at (0,0) (i1) {$i$};
\node at (1,0) (j1) {$j$};
\node at (2,0) (k1) {$k$,};
\node at (2.5,0) (i2) {$i$};
\node at (3.5,0) (j2) {$j$};
\node at (4.5,0) (k2) {$k$,};
\node at (5,0) (i3) {$i$};
\node at (6,0) (j3) {$j$};
\node at (7,0) (k3) {$k$,};
\node at (7.5,0) (or) {or};
\node at (8,0) (i4) {$i$};
\node at (9,0) (j4) {$j$};
\node at (10,0) (k4) {$k$};
\path[-angle 90]
	(i1) edge (j1)
	(k1) edge (j1)
	(k1) edge [bend left] (i1)
	(i2) edge (j2)
	(k2) edge (j2)
	(i2) edge [bend right] (k2)
	(j3) edge (k3)
	(j3) edge (i3)
	(i3) edge [bend right] (k3)
	(j4) edge (k4)
	(j4) edge (i4)
	(k4) edge [bend left] (i4);
\end{tikzpicture}\]
\end{itemize}
Together with \Cref{reducetomult1}, this proves claim (b).
\end{lem}
\begin{proof}
We use the same style of argument as for \Cref{edgecountlem}, by checking that certain seeds with principal or trivial coefficients are or are not fixed by the appropriate sequences of mutations. This is a routine check, so we merely state the necessary calculations.

For (i), it is sufficient to check that labelled seeds with quiver
\[\begin{tikzpicture}[scale=1.3]
\node at (0,0) (i1) {$i$};
\node at (1,0) (j1) {$j$};
\node at (2,0) (k1) {$k$};
\node at (2.5,0) (or) {or};
\node at (3,0) (i2) {$i$};
\node at (4,0) (j2) {$j$};
\node at (5,0) (k2) {$k$};
\path[-angle 90]
	(i1) edge (j1)
	(k1) edge (j1)
	(j2) edge (i2)
	(j2) edge (k2);
\end{tikzpicture}\]
are not fixed by $(\mu_i\mu_j\mu_k)^6$, and that labelled seeds with quiver
\[\begin{tikzpicture}[scale=1.5]
\node at (0,0) (i'1) {$\boxed{i'}$};
\node at (1,0) (j'1) {$\boxed{j'}$};
\node at (2,0) (k'1) {$\boxed{k'}$};
\node at (3,0) (i'2) {$\boxed{i'}$};
\node at (4,0) (j'2) {$\boxed{j'}$};
\node at (5,0) (k'2) {$\boxed{k'}$};
\node at (2.5,-0.5) (or) {or};
\node at (0,-1) (i1) {$i$};
\node at (1,-1) (j1) {$j$};
\node at (2,-1) (k1) {$k$};
\node at (3,-1) (i2) {$i$};
\node at (4,-1) (j2) {$j$};
\node at (5,-1) (k2) {$k$};
\path[-angle 90]
	(i'1) edge (i1)
	(j'1) edge (j1)
	(k'1) edge (k1)
	(i1) edge (j1)
	(j1) edge (k1)
	(i'2) edge (i2)
	(j'2) edge (j2)
	(k'2) edge (k2)
	(k2) edge (j2)
	(j2) edge (i2);
\end{tikzpicture}\]
are fixed by this mutation. As the first two quivers have trivial coefficients, no labelled seed with full subquiver $i\to j\leftarrow k$ or $i\leftarrow j\to k$ will be fixed by $(\mu_i\mu_j\mu_k)^6$, and as the second two quivers have principal coefficients, any labelled seed with full subquiver $i\to j\to k$ or $i\leftarrow j\leftarrow k$ will be fixed by $(\mu_i\mu_j\mu_k)^6$.

For (ii), we must check that labelled seeds with quiver
\[\begin{tikzpicture}[scale=1.2]
\node at (0,0) (i1) {$i$};
\node at (1,0) (j1) {$j$};
\node at (2,0) (k1) {$k$,};
\node at (2.5,0) (i2) {$i$};
\node at (3.5,0) (j2) {$j$};
\node at (4.5,0) (k2) {$k$,};
\node at (5,0) (i3) {$i$};
\node at (6,0) (j3) {$j$};
\node at (7,0) (k3) {$k$,};
\node at (7.5,0) (or) {or};
\node at (8,0) (i4) {$i$};
\node at (9,0) (j4) {$j$};
\node at (10,0) (k4) {$k$};
\path[-angle 90]
	(i1) edge (j1)
	(j1) edge (k1)
	(k1) edge [bend left] (i1)
	(i2) edge (j2)
	(j2) edge (k2)
	(i2) edge [bend right] (k2)
	(k3) edge (j3)
	(j3) edge (i3)
	(i3) edge [bend right] (k3)
	(k4) edge (j4)
	(j4) edge (i4)
	(k4) edge [bend left] (i4);
\end{tikzpicture}\]
are not fixed by $(\mu_i\mu_k\mu_i\mu_k\mu_i\mu_j)^2$, but labelled seeds with quiver
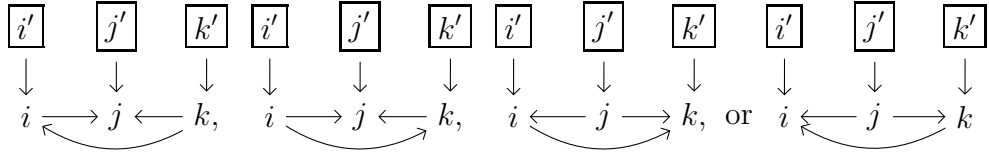
\[\begin{tikzpicture}[scale=1.2]
\node at (0,1) (i'1) {$\boxed{i'}$};
\node at (1,1) (j'1) {$\boxed{j'}$};
\node at (2,1) (k'1) {$\boxed{k'}$};
\node at (2.7,1) (i'2) {$\boxed{i'}$};
\node at (3.7,1) (j'2) {$\boxed{j'}$};
\node at (4.7,1) (k'2) {$\boxed{k'}$};
\node at (5.4,1) (i'3) {$\boxed{i'}$};
\node at (6.4,1) (j'3) {$\boxed{j'}$};
\node at (7.4,1) (k'3) {$\boxed{k'}$};
\node at (8.4,1) (i'4) {$\boxed{i'}$};
\node at (9.4,1) (j'4) {$\boxed{j'}$};
\node at (10.4,1) (k'4) {$\boxed{k'}$};
\node at (0,0) (i1) {$i$};
\node at (1,0) (j1) {$j$};
\node at (2,0) (k1) {$k$,};
\node at (2.7,0) (i2) {$i$};
\node at (3.7,0) (j2) {$j$};
\node at (4.7,0) (k2) {$k$,};
\node at (5.4,0) (i3) {$i$};
\node at (6.4,0) (j3) {$j$};
\node at (7.4,0) (k3) {$k$,};
\node at (7.9,0) (or) {or};
\node at (8.4,0) (i4) {$i$};
\node at (9.4,0) (j4) {$j$};
\node at (10.4,0) (k4) {$k$};
\path[-angle 90]
	(i1) edge (j1)
	(k1) edge (j1)
	(k1) edge [bend left] (i1)
	(i2) edge (j2)
	(k2) edge (j2)
	(i2) edge [bend right] (k2)
	(j3) edge (k3)
	(j3) edge (i3)
	(i3) edge [bend right] (k3)
	(j4) edge (k4)
	(j4) edge (i4)
	(k4) edge [bend left] (i4)
	(i'1) edge (i1)
	(i'2) edge (i2)
	(i'3) edge (i3)
	(i'4) edge (i4)
	(j'1) edge (j1)
	(j'2) edge (j2)
	(j'3) edge (j3)
	(j'4) edge (j4)
	(k'1) edge (k1)
	(k'2) edge (k2)
	(k'3) edge (k3)
	(k'4) edge (k4);
\end{tikzpicture}\]
are fixed by this mutation.
\end{proof}

To carry out the calculations necessary for the proof of \Cref{edgeorientlem}, we recommend using the Java applet \cite{kellerjava} of Keller, or the cluster algebra package in Sage \cite{musikercompendium}. In our experience, it is easiest to use the graphical interface of \cite{kellerjava} to check that the quiver is fixed, but to use Sage to verify that the cluster variables are fixed, as this is more computationally intensive. In fact, by \cite{gekhtmanproperties}*{Thm.~4}, it suffices to check that the cluster variables are fixed.

\Cref{edgecountlem,edgeorientlem} prove claims (a) and (b) respectively, completing the proof of \Cref{mainthm}.

\section{Examples}
\label{orbitgpegs}
Recall that, given a mutation class $\cS$ of labelled seeds, we can draw a labelled graph $\Delta=\lsgraph{\cS}$ with vertex set $\cS$ and an edge labelled $i$ between $s$ and $s\cdot\mu_i$ for all $s\in\cS$. Thus $\Delta$ encodes the data of the $\langle\mu_1,\dotsc,\mu_n\rangle$-action on $\cS$, but not the entire $M_n$-action. Given $\varphi\in\Aut_{M_n}(\cS)$, we have $\varphi(s\cdot\mu_i)=\varphi(s)\cdot\mu_i$ for all $s\in\cS$ and for all $i$, so $\varphi$ determines an automorphism of the labelled graph $\Delta$. Therefore, given any subgroup $W\leq\Aut_{M_n}(\cS)$, we get a quotient labelled graph $W\lquot\Delta$. This graph can also be constructed directly; its vertices are the $W$-orbits, and there is an edge labelled $i$ between $[s]$ and $[s]\cdot\mu_i=[s\cdot\mu_i]$ for each $W$-orbit $[s]$.
\begin{eg}
\label{a2graphs}
There are two labelled quivers in the mutation class of type $\typA2$, namely $1\to2$ and $1\from2$. 
Each is related to the other by mutation at either vertex, so
\[\dcW\lquot\Delta=\mathord{\begin{tikzpicture}[scale=1.5,baseline=0]
\node at (0,0.5) (t) {$\bullet$};
\node at (0,-0.5) (b) {$\bullet$};
\path[font=\scriptsize]
	(t) edge [bend left] node [right] {$2$} (b)
	(t) edge [bend right] node [left] {$1$} (b);
\end{tikzpicture}}\]
As each quiver is the opposite of the other, the two vertices are identified in $\cW\lquot\Delta$, and we have
\[\cW\lquot\Delta=\hspace{-1em}\mathord{\begin{tikzpicture}[every loop/.style={},scale=1.5,baseline=0]
\node at (0,0) (q) {$\bullet$};
\path[font=\scriptsize]
	(q) edge [in=125,out=55,loop] node [above] {$1$} (q)
	(q) edge [in=-125,out=-55,loop] node [below] {$2$} (q);
\end{tikzpicture}}\]
Recall from \Cref{a2eg} that in this case the graph $\Delta$ is a decagon;
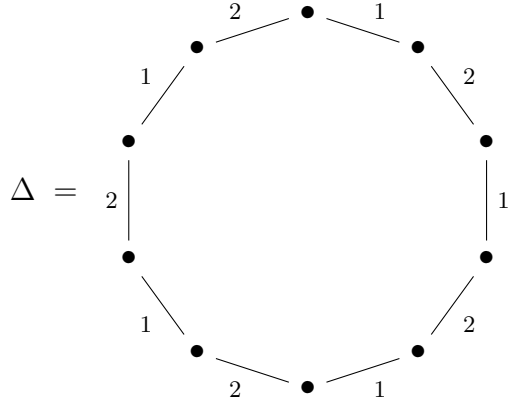
\[ 
\Delta \;=\; \mathord{
\begin{tikzpicture}[scale=1.25,baseline=0]
\foreach \n in {1,...,10}
 \node at (126-36*\n:2) (s\n) {$\bullet$};
\foreach \a/\b in {1/2, 3/4, 5/6, 7/8, 9/10}
\path[font=\scriptsize] (s\a) edge node [auto] {$1$} (s\b);
\foreach \a/\b in {2/3, 4/5, 6/7, 8/9, 10/1}
\path[font=\scriptsize] (s\a) edge node [auto] {$2$} (s\b);
\end{tikzpicture}}
\] 
Thus we see that the group $\dcW$ is the cyclic group $C_5$ generated by the rotation taking each vertex to the next but one clockwise, and $\cW$ is the entire automorphism group, which is in this case isomorphic to the dihedral group $D_5$ of symmetries of the pentagon; cf.\ \cite{assemcluster}*{Table~1}.
\end{eg}
\begin{eg}
We now consider the example of a quiver of type $\typA3$; while there are several choices of orientation, all of them are mutation equivalent because the underlying graph is a tree. There are $84$ labelled seeds in a cluster algebra of this type, so $\Delta$ is an $84$ vertex labelled graph. We find that
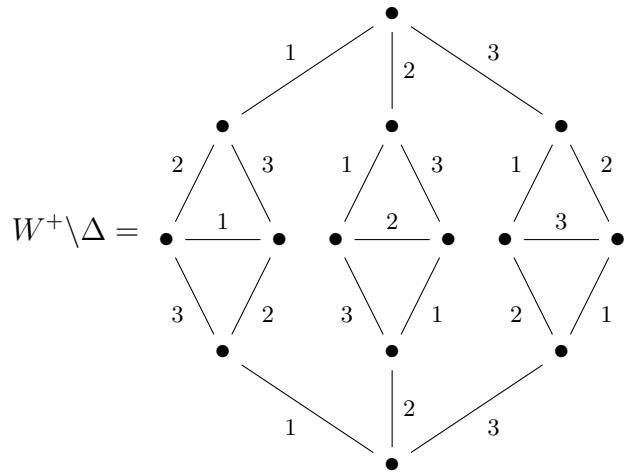
\[\dcW\lquot\Delta=\mathord{\begin{tikzpicture}[scale=1.5,baseline=0]
\node at (0,2) (1) {$\bullet$};
\node at (-1.5,1) (2) {$\bullet$};
\node at (0,1) (3) {$\bullet$};
\node at (1.5,1) (4) {$\bullet$};
\node at (-2,0) (5) {$\bullet$};
\node at (-1,0) (6) {$\bullet$};
\node at (-0.5,0) (7) {$\bullet$};
\node at (0.5,0) (8) {$\bullet$};
\node at (1,0) (9) {$\bullet$};
\node at (2,0) (10) {$\bullet$};
\node at (-1.5,-1) (11) {$\bullet$};
\node at (0,-1) (12) {$\bullet$};
\node at (1.5,-1) (13) {$\bullet$};
\node at (0,-2) (14) {$\bullet$};
\path[font=\scriptsize]
	(1) edge node [above left] {$1$} (2)
	(1) edge node [right] {$2$} (3)
	(1) edge node [above right] {$3$} (4)
	(2) edge node [above left] {$2$} (5)
	(2) edge node [above right] {$3$} (6)
	(3) edge node [above left] {$1$} (7)
	(3) edge node [above right] {$3$} (8)
	(4) edge node [above left] {$1$} (9)
	(4) edge node [above right] {$2$} (10)
	(5) edge node [above] {$1$} (6)
	(7) edge node [above] {$2$} (8)
	(9) edge node [above] {$3$} (10)
	(5) edge node [below left] {$3$} (11)
	(6) edge node [below right] {$2$} (11)
	(7) edge node [below left] {$3$} (12)
	(8) edge node [below right] {$1$} (12)
	(9) edge node [below left] {$2$} (13)
	(10) edge node [below right] {$1$} (13)
	(11) edge node [below left] {$1$} (14)
	(12) edge node [right] {$2$} (14)
	(13) edge node [below right] {$3$} (14);
\end{tikzpicture}}\]
where the highest and lowest vertices are the two $3$-cycles in the mutation class, and
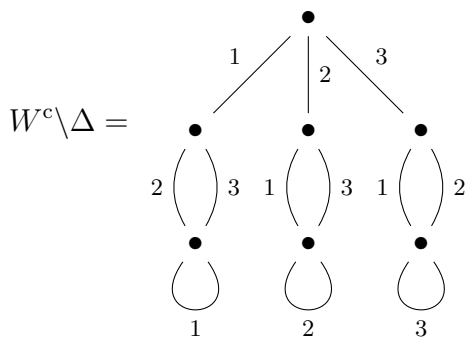
\[\cW\lquot\Delta=\mathord{\begin{tikzpicture}[every loop/.style={},scale=1.5,baseline=0]
\node at (0,1) (1) {$\bullet$};
\node at (-1,0) (2) {$\bullet$};
\node at (0,0) (3) {$\bullet$};
\node at (1,0) (4) {$\bullet$};
\node at (-1,-1) (5) {$\bullet$};
\node at (0,-1) (6) {$\bullet$};
\node at (1,-1) (7) {$\bullet$};
\path[font=\scriptsize]
	(1) edge node [above left] {$1$} (2)
	(1) edge node [right] {$2$} (3)
	(1) edge node [above right] {$3$} (4)
	(2) edge [bend right] node [left] {$2$} (5)
	(2) edge [bend left] node [right] {$3$} (5)
	(3) edge [bend right] node [left] {$1$} (6)
	(3) edge [bend left] node [right] {$3$} (6)
	(4) edge [bend right] node [left] {$1$} (7)
	(4) edge [bend left] node [right] {$2$} (7)
	(5) edge [in=-125,out=-55,loop] node [below] {$1$} (5)
	(6) edge [in=-125,out=-55,loop] node [below] {$2$} (6)
	(7) edge [in=-125,out=-55,loop] node [below] {$3$} (7);
\end{tikzpicture}}\]
We can see directly in this example (and the previous one) that $\cW\lquot\Delta$ has no non-trivial automorphisms as a labelled graph. Hence the group $\Aut_{M_n}(\cW\lquot\cS)$, consisting of automorphisms of $\cW\lquot\Delta$ commuting with the permutation action on vertices, is also trivial. This is consistent with our earlier observations; it follows from \Cref{autXmodWnormal} that $\cW$ is normal in $\Aut_{M_n}(\cS)$, and in fact we have $\cW=\Aut_{M_n}(\cS)$ by \Cref{autdeterminesquiv}.
\end{eg}
\begin{eg}
We now consider an example of infinite type, namely a non-cyclic orientation of $\afftypA{2}$.
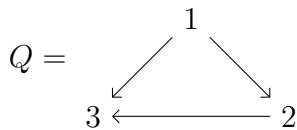
\[Q=\mathord{\begin{tikzpicture}[scale=1.3,baseline=0]
\node at (0,0.5) (1) {$1$};
\node at (-1,-0.5) (3) {$3$};
\node at (1,-0.5) (2) {$2$};
\path[-angle 90]
	(1) edge (3)
	(2) edge (3)
	(1) edge (2);
\end{tikzpicture}}\]
This quiver defines a cluster-infinite cluster algebra, so the graph $\Delta$ is infinite. However, it has finite mutation type, so the quotients $\dcW\lquot\Delta$ and $\cW\lquot\Delta$ are both finite. We have
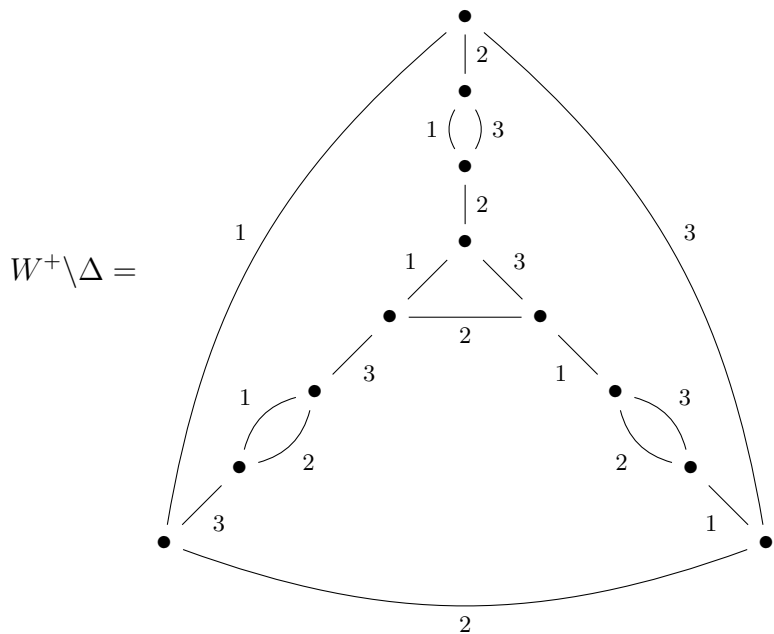
\[\dcW\lquot\Delta=\mathord{\begin{tikzpicture}[baseline=0]
\node at (0,3.5) (1) {$\bullet$};
\node at (0,2.5) (2) {$\bullet$};
\node at (0,1.5) (3) {$\bullet$};
\node at (0,0.5) (4) {$\bullet$};
\node at (-1,-0.5) (5) {$\bullet$};
\node at (1,-0.5) (6) {$\bullet$};
\node at (-2,-1.5) (7) {$\bullet$};
\node at (2,-1.5) (8) {$\bullet$};
\node at (-3,-2.5) (9) {$\bullet$};
\node at (3,-2.5) (10) {$\bullet$};
\node at (-4,-3.5) (11) {$\bullet$};
\node at (4,-3.5) (12) {$\bullet$};
\path[font=\scriptsize]
	(1) edge [bend right=20] node [above left] {$1$} (11)
	(1) edge node [right] {$2$} (2)
	(1) edge [bend left=20] node [above right] {$3$} (12)
	(2) edge [bend right] node [left] {$1$} (3)
	(2) edge [bend left] node [right] {$3$} (3)
	(3) edge node [right] {$2$} (4)
	(4) edge node [above left] {$1$} (5)
	(4) edge node [above right] {$3$} (6)
	(5) edge node [below] {$2$} (6)
	(5) edge node [below right] {$3$} (7)
	(6) edge node [below left] {$1$} (8)
	(7) edge [bend right] node [above left] {$1$} (9)
	(7) edge [bend left] node [below right] {$2$} (9)
	(8) edge [bend left] node [above right] {$3$} (10)
	(8) edge [bend right] node [below left] {$2$} (10)
	(9) edge node [below right] {$3$} (11)
	(10) edge node [below left] {$1$} (12)
	(11) edge [bend right=20] node [below] {$2$} (12);
\end{tikzpicture}}\]
where the uppermost vertex is $Q$, and
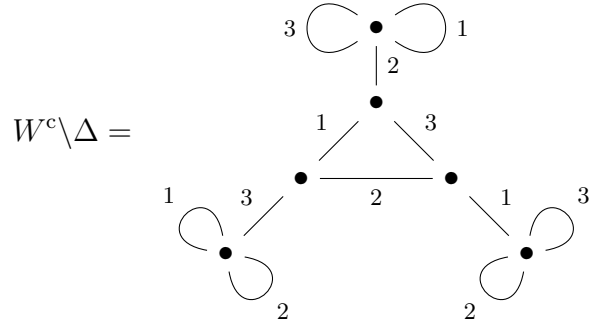
\[\cW\lquot\Delta=\mathord{\begin{tikzpicture}[every loop/.style={},baseline=0]
\node at (0,1.5) (1) {$\bullet$};
\node at (0,0.5) (2) {$\bullet$};
\node at (-1,-0.5) (3) {$\bullet$};
\node at (1,-0.5) (4) {$\bullet$};
\node at (-2,-1.5) (5) {$\bullet$};
\node at (2,-1.5) (6) {$\bullet$};
\path[font=\scriptsize]
	(1) edge [out=35,in=-35,loop] node [right] {$1$} (1)
	(1) edge [out=145,in=-145,loop] node [left] {$3$} (1)
	(1) edge node [right] {$2$} (2)
	(2) edge node [above left] {$1$} (3)
	(2) edge node [above right] {$3$} (4)
	(3) edge node [below] {$2$} (4)
	(3) edge node [above left] {$3$} (5)
	(4) edge node [above right] {$1$} (6)
	(5) edge [out=100,in=170,loop] node [above left] {$1$} (5)
	(5) edge [out=-10,in=-80,loop] node [below right] {$2$} (5)
	(6) edge [out=10,in=80,loop] node [above right] {$3$} (6)
	(6) edge [out=-100,in=-170,loop] node [below left] {$2$} (6);
\end{tikzpicture}}\]
As in the previous two examples, $\cW\lquot\Delta$ has no non-trivial automorphisms as a labelled graph, so $\Aut_{M_n}(\cW\lquot\cS)=\set{1}$.
\end{eg}
\begin{eg}\label{evilquiver}
Next, we consider an infinite type example with infinite mutation class, namely the quiver
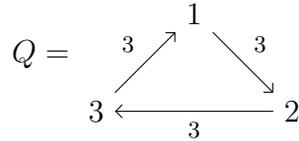
\[Q=\mathord{\begin{tikzpicture}[scale=1.3,baseline=0]
\node at (0,0.5) (1) {$1$};
\node at (-1,-0.5) (3) {$3$};
\node at (1,-0.5) (2) {$2$};
\path[-angle 90,font=\scriptsize]
	(1) edge node [above right] {$3$} (2)
	(2) edge node [below] {$3$} (3)
	(3) edge node [above left] {$3$} (1);
\end{tikzpicture}}\]
The triples of edge multiplicities occurring in the quivers in this class are all solutions to the Markov-type equation $x^2+y^2+z^2=xyz$, and mutations correspond to passing to a neighbour (see \cite{casselsintroduction}*{\S II.3}); if $(a,b,c)$ is a solution, then so are $(a,b,ab-c)$, $(a,ac-b,c)$ and $(bc-a,b,c)$.

It follows that any sequence of mutations starting from $Q$ (such that successive mutations are at distinct vertices) increases the maximal multiplicity of an arrow in the quiver, and thus no quiver occurs twice in the same component of $\Delta$, even up to similarity. The graph $\Delta$ has six components, one for each permutation of the labels of the initial seed, each of which is a $3$-regular tree. Each cluster automorphism is determined by a permutation of the initial cluster, and thus $\cW\isom\cAut{\cA}\isom\Sym{3}$. The direct cluster automorphisms must preserve the cyclic ordering of the initial labelled cluster, and so $\dcW\isom\dircAut{\cA}\isom C_3$ is cyclic of order $3$.

The graph $\cW\lquot\Delta$ is the infinite $3$-regular labelled tree
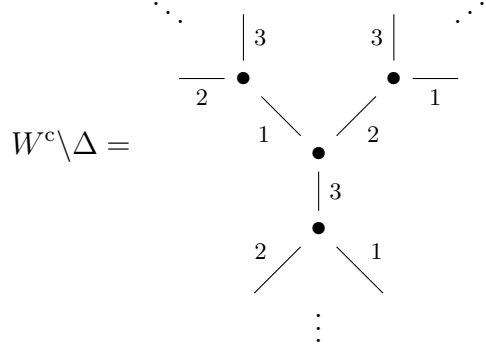
\[\cW\lquot\Delta=\mathord{\begin{tikzpicture}[baseline=0]
\node at (-2,2) (dots1) {$\ddots$};
\node at (-1,2) (up1) {};
\node at (1,2) (up2) {};
\node at (2,2) (dots2) {$\iddots$};
\node at (-2,1) (left1) {};
\node at (-1,1) (1) {$\bullet$};
\node at (1,1) (2) {$\bullet$};
\node at (2,1) (right1) {};
\node at (0,0) (3) {$\bullet$};
\node at (0,-1) (4) {$\bullet$};
\node at (-1,-2) (left2) {};
\node at (1,-2) (right2) {};
\node at (0,-2.2) (dots3) {$\vdots$};
\path[font=\scriptsize]
	(1) edge node [right] {$3$} (up1)
	(1) edge node [below] {$2$} (left1)
	(1) edge node [below left] {$1$} (3)
	(2) edge node [left] {$3$} (up2)
	(2) edge node [below] {$1$} (right1)
	(2) edge node [below right] {$2$} (3)
	(3) edge node [right] {$3$} (4)
	(4) edge node [above left] {$2$} (left2)
	(4) edge node [above right] {$1$} (right2);
\end{tikzpicture}}\]
which still has automorphisms as a labelled graph. However, only the trivial automorphism commutes with the permutation action; the only similarity class of seeds stabilised by all permutations is the class $[s]$ of those with quiver $Q$ or $Q^{\op}$, so any $\varphi\in\Aut_{M_n}(\cW\lquot\cS)$ satisfies $\varphi([s])=[s]$, and so $\Aut_{M_n}(\cW\lquot\cS)=\set{1}$.

However, we will now show that there are automorphisms of $\cS$, commuting with the $M_n$-action, that do not lie in $\cW$. Let $s\in\cS$ and consider $g\in\Stab_{M_n}(s)$. We may write $g=\mu_{i_1}\dotsm\mu_{i_k}\sigma$, where $i_1,\dotsc,i_k\in I$ and $\sigma\in\Sym{n}$. If $\sigma$ is not the identity, then $s\cdot g$ lies in a different component of $\Delta$ to $s$, contradicting $g\in\Stab_{M_n}(s)$. No seed is fixed by a non-trivial sequence of mutations, so in fact we must have $g=1$. Thus every seed of $\cS$ has trivial stabiliser, so $M_n$ acts freely and transitively on $\cS$, and $\Aut_{M_n}(\cS)\isom M_n$ is infinite. In this case, $\cW$ is finite, so $\cW\ne\Aut_{M_n}(\cS)$. This example shows that the conclusion of \Cref{mainthm} may not hold for mutation classes that are not small.

In this example, $\cW$ is not normal in $A$. Indeed, by \Cref{autXmodWnormal}, there is an injection
\[\norm{A}{\cW}/\cW\incl\Aut_{M_n}(\cW\lquot\cS)=\{1\},\]
so $\norm{A}{\cW}=\cW\ne A$. It is also straightforward to construct explicit examples of automorphisms that do not normalise $\cW$. As $\Stab_{M_n}(s_1)=\Stab_{M_n}(s_2)$ for any $s_1,s_2\in\cS$, the map $\varphi\colon s_1\cdot g\mapsto s_2\cdot g$ is always a well-defined automorphism of $\cS$. Let $s$ be a seed with quiver $Q$, and let $w(s\cdot g)=s\cdot\sigma g$, where $\sigma$ is the permutation $\tran{1}{2}$. As $Q\cdot\sigma\approx Q$, the seeds $s\cdot g$ and $s\cdot\sigma g$ have similar quivers for all $g$, and so $w\in\cW$. Now let $\varphi(s\cdot g)=s\cdot\mu_1g$. Then
\[\varphi w\varphi^{-1}(s)=s\cdot\mu_1\tran{1}{2}\mu_1=s\cdot\mu_1\mu_2\tran{1}{2}\]
is not similar to $s$, so $\varphi w\varphi^{-1}\notin\cW$.
\end{eg}

In each of these examples, we may observe that $\cW/\dcW\isom C_2$ is cyclic of order $2$. This is not a general phenomenon, but by \cite{assemcluster}*{Thm.~2.11} is equivalent to the fact that each quiver $Q$ is mutation equivalent to some $Q'$ isomorphic to $Q^{\op}$ as an unlabelled quiver.

\bibliographystyle{amsalpha}

\bibliography{../../mainbib}

\begin{bibdiv}
\begin{biblist}

\bib{assemcluster}{article}{
      author={Assem, Ibrahim},
      author={Schiffler, Ralf},
      author={Shramchenko, Vasilisa},
       title={Cluster automorphisms},
        date={2012},
        ISSN={0024-6115},
     journal={Proc. Lond. Math. Soc. (3)},
      volume={104},
      number={6},
       pages={1271\ndash 1302},
         url={http://dx.doi.org/10.1112/plms/pdr049},
        note={\href{http://arxiv.org/abs/1009.0742}{arXiv:1009.0742}
  [math.RT]},
      review={\MR{2946087}},
}

\bib{berensteincluster3}{article}{
      author={Berenstein, Arkady},
      author={Fomin, Sergey},
      author={Zelevinsky, Andrei},
       title={Cluster algebras. {III}. {U}pper bounds and double {B}ruhat
  cells},
        date={2005},
        ISSN={0012-7094},
     journal={Duke Math. J.},
      volume={126},
      number={1},
       pages={1\ndash 52},
         url={http://dx.doi.org/10.1215/S0012-7094-04-12611-9},
        note={\href{http://arxiv.org/abs/math.RT/0305434}{arXiv:math/0305434}
  [math.RT]},
      review={\MR{2110627 (2005i:16065)}},
}

\bib{browntopology}{book}{
      author={Brown, Ronald},
       title={Topology and groupoids},
   publisher={BookSurge, LLC, Charleston, SC},
        date={2006},
        ISBN={1-4196-2722-8},
        note={Third edition of {\it Elements of modern topology} [McGraw--Hill,
  New York, 1968;
  MR\href{http://www.ams.org/mathscinet-getitem?mr=227979}{0227979}]},
      review={\MR{2273730}},
}

\bib{berensteinquantum}{article}{
      author={Berenstein, Arkady},
      author={Zelevinsky, Andrei},
       title={Quantum cluster algebras},
        date={2005},
        ISSN={0001-8708},
     journal={Adv. Math.},
      volume={195},
      number={2},
       pages={405\ndash 455},
         url={http://dx.doi.org/10.1016/j.aim.2004.08.003},
        note={\href{http://arxiv.org/abs/math/0404446}{arXiv:math/0404446}
  [math.QA]},
      review={\MR{2146350 (2006a:20092)}},
}

\bib{casselsintroduction}{book}{
      author={Cassels, John W.~S.},
       title={An introduction to {D}iophantine approximation},
      series={Cambridge Tracts in Mathematics and Mathematical Physics, No.
  45},
   publisher={Cambridge University Press},
     address={New York},
        date={1957},
      review={\MR{0087708 (19,396h)}},
}

\bib{cerullilinear}{article}{
      author={{Cerulli Irelli}, Giovanni},
      author={Keller, Bernhard},
      author={{Labardini-Fragoso}, Daniel},
      author={Plamondon, Pierre-Guy},
       title={Linear independence of cluster monomials for skew-symmetric
  cluster algebras},
        date={2013},
        ISSN={0010-437X},
     journal={Compos. Math.},
      volume={149},
      number={10},
       pages={1753\ndash 1764},
         url={http://dx.doi.org.ezproxy1.bath.ac.uk/10.1112/S0010437X1300732X},
        note={\href{http://arxiv.org/abs/1203.1307}{arXiv:1203.1307}
  [math.RT]},
      review={\MR{3123308}},
}

\bib{derksennewgraphs}{article}{
      author={Derksen, Harm},
      author={Owen, Theodore},
       title={New graphs of finite mutation type},
        date={2008},
        ISSN={1077-8926},
     journal={Electron. J. Combin.},
      volume={15},
      number={1},
       pages={Research Paper 139, 15},
         url={http://www.combinatorics.org/Volume_15/Abstracts/v15i1r139.html},
        note={\href{http://arxiv.org/abs/0804.0787}{arXiv:0804.0787}
  [math.CO]},
      review={\MR{2465763 (2009j:05102)}},
}

\bib{fockcluster1}{article}{
      author={Fock, Vladimir~V.},
      author={Goncharov, Alexander~B.},
       title={Cluster ensembles, quantization and the dilogarithm},
        date={2009},
        ISSN={0012-9593},
     journal={Ann. Sci. \'Ec. Norm. Sup\'er. (4)},
      volume={42},
      number={6},
       pages={865\ndash 930},
         url={http://dx.doi.org/10.1007/978-0-8176-4745-2_15},
        note={\href{http://arxiv.org/abs/math/0311245}{arXiv:math/0311245}
  [math.AG]},
      review={\MR{2567745 (2011f:53202)}},
}

\bib{fomintriangulated1}{article}{
      author={Fomin, Sergey},
      author={Shapiro, Michael},
      author={Thurston, Dylan},
       title={Cluster algebras and triangulated surfaces. {I}. {C}luster
  complexes},
        date={2008},
        ISSN={0001-5962},
     journal={Acta Math.},
      volume={201},
      number={1},
       pages={83\ndash 146},
         url={http://dx.doi.org/10.1007/s11511-008-0030-7},
        note={\href{http://arxiv.org/abs/math/0608367}{arXiv:math/0608367}
  [math.RA]},
      review={\MR{2448067 (2010b:57032)}},
}

\bib{fultonalgebraic}{book}{
      author={Fulton, William},
       title={Algebraic topology: A first course},
      series={Graduate Texts in Mathematics},
   publisher={Springer--Verlag},
     address={New York},
        date={1995},
      volume={153},
        ISBN={0-387-94326-9; 0-387-94327-7},
         url={http://dx.doi.org/10.1007/978-1-4612-4180-5},
      review={\MR{1343250 (97b:55001)}},
}

\bib{fomincluster1}{article}{
      author={Fomin, Sergey},
      author={Zelevinsky, Andrei},
       title={Cluster algebras. {I}. {F}oundations},
        date={\noop{b}2002},
        ISSN={0894-0347},
     journal={J. Amer. Math. Soc.},
      volume={15},
      number={2},
       pages={497\ndash 529 (electronic)},
         url={http://dx.doi.org/10.1090/S0894-0347-01-00385-X},
        note={\href{http://arxiv.org/abs/math/0104151}{arXiv:math/0104151}
  [math.RT]},
      review={\MR{1887642 (2003f:16050)}},
}

\bib{fomincluster2}{article}{
      author={Fomin, Sergey},
      author={Zelevinsky, Andrei},
       title={Cluster algebras. {II}. {F}inite type classification},
        date={\noop{b}2003},
        ISSN={0020-9910},
     journal={Invent. Math.},
      volume={154},
      number={1},
       pages={63\ndash 121},
         url={http://dx.doi.org/10.1007/s00222-003-0302-y},
        note={\href{http://arxiv.org/abs/math/0208229}{arXiv:math/0208229}
  [math.RA]},
      review={\MR{2004457 (2004m:17011)}},
}

\bib{fomincluster4}{article}{
      author={Fomin, Sergey},
      author={Zelevinsky, Andrei},
       title={Cluster algebras. {IV}. {C}oefficients},
        date={\noop{b}2007},
        ISSN={0010-437X},
     journal={Compos. Math.},
      volume={143},
      number={1},
       pages={112\ndash 164},
         url={http://dx.doi.org/10.1112/S0010437X06002521},
        note={\href{http://arxiv.org/abs/math.RA/0602259}{arXiv:math/0602259}
  [math.RA]},
      review={\MR{2295199 (2008d:16049)}},
}

\bib{gekhtmanproperties}{article}{
      author={Gekhtman, Michael},
      author={Shapiro, Michael},
      author={Vainshtein, Alek},
       title={On the properties of the exchange graph of a cluster algebra},
        date={2008},
        ISSN={1073-2780},
     journal={Math. Res. Lett.},
      volume={15},
      number={2},
       pages={321\ndash 330},
        note={\href{http://arxiv.org/abs/math/0703151}{arXiv:math/0703151}
  [math.CO]},
      review={\MR{2385644 (2009b:13010)}},
}

\bib{kellercluster}{incollection}{
      author={Keller, Bernhard},
       title={Cluster algebras, quiver representations and triangulated
  categories},
        date={2010},
   booktitle={Triangulated categories},
      series={London Math. Soc. Lecture Note Ser.},
      volume={375},
   publisher={Cambridge Univ. Press},
     address={Cambridge},
       pages={76\ndash 160},
        note={\href{http://arxiv.org/abs/0807.1960}{arXiv:0807.1960}
  [math.RT]},
      review={\MR{2681708 (2011h:13033)}},
}

\bib{kellerjava}{misc}{
      author={Keller, Bernhard},
       title={Quiver mutation in {J}ava},
        note={Java applet available at
  \url{http://www.math.jussieu.fr/~keller/quivermutation/}},
}

\bib{krammergarside}{article}{
      author={Krammer, Daan},
       title={A class of {G}arside groupoid structures on the pure braid
  group},
        date={2008},
        ISSN={0002-9947},
     journal={Trans. Amer. Math. Soc.},
      volume={360},
      number={8},
       pages={4029\ndash 4061},
         url={http://dx.doi.org/10.1090/S0002-9947-08-04313-4},
        note={\href{http://arxiv.org/abs/math/0509165}{arXiv:math/0509165}
  [math.GR]},
      review={\MR{2395163 (2010a:20085)}},
}

\bib{ladkaniwhich}{unpublished}{
      author={Ladkani, Sefi},
       title={Which mutation classes of quivers have constant number of
  arrows?},
        date={2011},
        note={\href{http://arxiv.org/abs/1104.0436v1}{arXiv:1104.0436v1}
  [math.CO]},
}

\bib{musikercompendium}{article}{
      author={Musiker, Gregg},
      author={Stump, Christian},
       title={A compendium on the cluster algebra and quiver package in
  \tt{{S}age}},
        date={2010/12},
        ISSN={1286-4889},
     journal={S\'em. Lothar. Combin.},
      volume={65},
       pages={Art. B65d},
        note={\href{http://arxiv.org/abs/1102.4844}{arXiv:1102.4844}
  [math.CO]},
      review={\MR{2862990 (2012m:13045)}},
}

\bib{reitencluster}{inproceedings}{
      author={Reiten, Idun},
       title={Cluster categories},
        date={2010},
   booktitle={Proceedings of the {I}nternational {C}ongress of
  {M}athematicians. {V}olume {I}},
   publisher={Hindustan Book Agency, New Delhi},
       pages={558\ndash 594},
        note={\href{http://arxiv.org/abs/1012.4949}{arXiv:1012.4949}
  [math.RT]},
      review={\MR{2827905 (2012g:16002)}},
}

\end{biblist}
\end{bibdiv}

\end{document}